\documentclass[12pt,reqno]{amsart}
\usepackage[margin=1in]{geometry}
\usepackage{amsmath,amssymb,amsthm,graphicx,amsxtra, setspace,lineno}
\usepackage{mathrsfs}
\usepackage{url}
\usepackage{cite}
\usepackage{hyperref}
\usepackage{mathtools}
\usepackage{dsfont}

\usepackage{comment}
\allowdisplaybreaks

\renewcommand{\div}{\operatorname{div}}
\newcommand{\divv}{\mathbf{div}}

\DeclareMathOperator*{\esssup}{ess\,sup}

\newtheorem{conditions}{Conditions}[section]
\newtheorem{theorem}{Theorem}[section]

\newtheorem{definition}{Definition}[section]
\newtheorem{lemma}{Lemma}[section]

\numberwithin{equation}{section} \numberwithin{theorem}{section}
\numberwithin{lemma}{section} \numberwithin{conditions}{section}
\numberwithin{corollary}{section} \numberwithin{definition}{section}
\numberwithin{remark}{section} \numberwithin{proposition}{section}

\newcommand{\0}{\boldsymbol{0}}
\newcommand{\uu}{\mathbf{u}}
\newcommand{\vv}{\mathbf{v}}
\newcommand{\ww}{\mathbf{w}}
\newcommand{\ff}{\mathbf{f}}
\newcommand{\xx}{\mathbf{x}}

\newcommand{\mm}{\mathbf{m}}
\newcommand{\Ii}{\mathbf{I}}
\newcommand{\bM}{\mathbf{M}}
\newcommand{\bT}{\mathbf{T}}
\newcommand{\Du}{\mathbf{D}}
\newcommand{\Qq}{\mathbf{Q}}
\newcommand{\Tau}{\boldsymbol{\tau}}
\newcommand{\Sig}{\boldsymbol{\sigma}}
\renewcommand{\div}{\operatornamewithlimits{div}}

\renewcommand{\L}{\mathrm{L}}
\newcommand{\C}{\mathrm{C}}
\newcommand{\W}{\mathrm{W}}
\renewcommand{\H}{\mathrm{H}}
\newcommand{\LL}{\mathbf{L}}
\newcommand{\CC}{\mathbf{C}}
\newcommand{\WW}{\mathbf{W}}
\newcommand{\HH}{\mathbf{H}}
\newcommand{\VV}{\mathbf{V}}

\usepackage{bm}
\usepackage{enumitem}

\usepackage[margin=1in]{geometry}
\let\origmaketitle\maketitle
\def\maketitle{
  \begingroup
  \def\uppercasenonmath##1{} 
  \let\MakeUppercase\bfseries 
  \origmaketitle
  \endgroup
}

\makeatletter \@ifdefinable\@latex@chi{\let\@latex@chi\chi}
\renewcommand*\chi{{\@latex@chi\smash[t]{\mathstrut}}} 
\makeatletter

\newcommand{\Addresses}{{
        \footnote{
            \noindent
            \textsuperscript{(1)}Lavrentyev Institute of Hydrodynamics, Siberian Branch of Russian Academy of Sciences, Novosibirsk, Russian Federation. \\
            \textsuperscript{(2)}FCT - Universidade do Algarve, Faro, Portugal.  \\
            \textsuperscript{(3)}CIDMA - Universidade de Aveiro, Aveiro, Portugal.  \\
            \textsuperscript{(4)}Laboratory for Mathematical and Computer Modeling in Natural and Industrial Systems, Altai State University, Barnaul, Russian Federation. \\
            \textsuperscript{(5)}Al-Farabi Kazakh National University, Almaty, Kazakhstan.
             \par\nopagebreak

            \noindent  \textit{E-mails:} S.N.~Antontsev: \url{antontsevsn@mail.ru}, H.B.~de~Oliveira: \url{holivei@ualg.pt}, I.V.~Kuznetsov: \url{kuznetsov_i@hydro.nsc.ru}, D.A.~Prokudin: \url{prokudin@hydro.nsc.ru}, Kh.~Khompysh: \url{konat_k@mail.ru}

}}}

\usepackage{amsrefs}

\begin{document}%

\title[]{Mixtures of nonhomogeneous viscoelastic incompressible fluids governed by the Kelvin-Voigt equations\Addresses}
\author[]{S.N.~Antontsev\textsuperscript{(1)}, H.B.~de~Oliveira\textsuperscript{(2,3)}, I.V.~Kuznetsov\textsuperscript{(1,4)}, \\ D.A.~Prokudin\textsuperscript{(1,4)} and Kh.~Khompysh\textsuperscript{(5)}}

    \maketitle

\begin{abstract}
An initial-and boundary-value problem for the Kelvin–Voigt system, modeling a
mixture of $n$ incompressible and viscoelastic fluids, with non-constant density, is investigated in this work.
The existence of global-in-time weak solutions is established: velocity, density and pressure.
Under additional regularity assumptions, we also prove the uniqueness of the solution.
\end{abstract}

\newcommand{\wt}{{\widetilde t}}
\newcommand{\wu}{{\widetilde u}}


\section{Introduction}

The aim of this work is to study the following system of partial
differential equations of Kelvin-Voigt type,
\begin{align}
\label{eq.1.26} & \mathrm{div}\mathbf{v}_i=0,
\\
\label{eq.1.27} &
\partial_t(\rho_i\mathbf{v}_i)+
\mathbf{div}\left(\rho_i\mathbf{v}_i\otimes\mathbf{v}_i\right)=\rho_i\mathbf{f} _i-\nabla\pi_i+
\sum\limits_{j=1}^n\big(\mu_{ij}\Delta\mathbf{v}_j+\kappa_{ij}\partial_t\Delta\mathbf{v}_j
+\gamma_{ij}\left(\mathbf{v}_j-\mathbf{v}_i\right)\big),
\\
\label{eq.1.28} & \partial_t\rho_i+\mathrm{div} (\rho_i\mathbf{v}_i)=0,\qquad
\rho_i>0,
\end{align}
for $i\in\{1,\dots,n\}$. This system of equations is considered in a
cylinder $Q_{T}:=\Omega\times(0,T]$, where $\Omega$ is a bounded
domain of $\mathbb{R}^{d}$, $d\in \{2,3\}$, and $T$ is a given
positive constant. The notation used here is the following:
$\mathbf{v}_1=(v_1^1,\dots,v_1^d)$, \dots, $\mathbf{v}_n=(v_n^1,\dots,v_n^d)$,
$\rho_1$, \dots, $\rho_n$, and $\pi_1$, \dots, $\pi_n$ are the
sought solutions, whereas $\mathbf{f}_1=(f_1^1,\dots,f_1^d)$, \dots,
$\mathbf{f}_n=(f_n^1,\dots,f_n^d)$ are given functions, and $\mu_{ij}$,
$\kappa_{ij}$, $\gamma_{ij}$, with $i,j\in\{1,\dots,n\}$, are
positive constants. For $i\in\{1,\ldots,n\}$, we supplement
\eqref{eq.1.26}--\eqref{eq.1.28} with the following initial
conditions in $\Omega$,
\begin{equation}
\label{eq.1.29} \rho_i(\mathbf{x},0)=\rho_{i,0}(\mathbf{x}),\qquad
\rho_i(\mathbf{x},0)\mathbf{v}_i(\mathbf{x},0)=\rho_{i,0}(\mathbf{x})\mathbf{v}_{i,0}(\mathbf{x})\,,
\end{equation}
and with the following zero Dirichlet boundary conditions on
$\Gamma_{T}:=\partial\Omega\times(0, T]$,
\begin{equation}
\label{eq.1.30} \mathbf{v}_i=\boldsymbol{0}\,,
\end{equation}
where $\partial\Omega$ denotes the boundary of the domain $\Omega$.

Problem formed by the system of equations
\eqref{eq.1.26}--\eqref{eq.1.30} can be used to model flows of
mixtures formed by $n$ incompressible viscoelastic fluids with
non-constant densities.
A fundamental assumption in the theory of mixtures is that each
infinitesimal point of the continuum is occupied only by molecules
of a single constituent. Although such an assumption is physically
impossible, it is adequate if we assume that each constituent is
sufficiently dense, in the sense that each infinitesimal point can
be considered as a continuum of only one constituent. Whenever the
mixture deforms, these coexistence continua will deform relative to
each other~\cite{AC:1976a,AC:1976b}. By extending the principles of
continuum mechanics for single-constituent fluids, making certain
simplifications and assumptions, known in the literature as
metaphysical principles~\cite{Truesdell:1984}, we can derive the
governing equations for each constituent of the mixture.

In this work, we extend the results obtained in \cite{AKPS} to the case of a homogeneous $n$-constituent mixture. Our analysis is based on the framework developed for global weak solutions to the Kelvin-Voigt system, which models incompressible fluids with variable density, as investigated in \cite{AOKh:2019:c, Antontsev-2021, deOliveira}. The classical formulation of the Kelvin-Voigt equations describes the behavior of single-constituent viscoelastic fluids. The existence of global-in-time weak solutions has already been established, along with uniqueness results under appropriate additional regularity conditions~\cite{Oskolkov:1985,Oskolkov:1989,Zvyagin2010,deOliveira,Antontsev-2021}.
{Recently, 
the nonhomogeneous Kelvin-Voigt equations coupled with the convective nonhomogeneous Cahn-Hilliard equation were studied in \cite{Temam:2023}.}

The present paper is organized as follows.
Section~\ref{Sect:GE} is devoted to justifying the governing equations, explaining some extensions and limitations of the considered problem.
In Section~\ref{Sect:EWS}, we establish the global-in-time existence of a weak solutions (velocity and density) to the problem \eqref{eq.1.26}--\eqref{eq.1.30}.
The proof follows by combining the Galerkin approach with compactness arguments and the monotonicity method.
That a constructed weak solution (velocity and density) uniquely determines the pressure, is proved in
Section~\ref{Sect:Rec-Press}. In Section~\ref{Sect:Regularity}, the regularity of
weak solutions is improved. The smoothness of the solution proved
here allow us to prove its uniqueness, which, in turn, is aim of Section~\ref{Sect:Unique}.

The notation used in this work is quite standard in this field -- we address the interested reader to some of the monographs cited hereinafter~\cite{Temam:1984,Galdi:2011,Mazya:2011,Raja-Tao:1995}.
Lowercase boldface letters denote vector-valued functions and non-boldface letters stay for scalars.
Boldface capital letters denote tensors, or matrixes.
This notation is also extended for function spaces of scalar or vector-valued and tensor-valued function spaces.
The letters $C$ and $K$ will always denote positive constants, whose value may change from line to line, and its dependence on other parameters or data will always be clear from the exposition.
By $\Omega$, we always denote a domain (open and connected subset) in $\mathbb{R}^d$, with $d\geq 2$ integer.
The notation $\overline{\Omega}$ stands for the closure of the domain $\Omega$ in $\mathbb{R}^d$.
When $\alpha\in(0,1)$, $\C^{0,\alpha}(\overline{\Omega})$ denotes the space of function that are H\"older-continuous with exponent $\alpha$.
The norm in $\C^{0,\alpha}(\overline{\Omega})$ is denoted by $\|\cdot\|_{0,\alpha}$.
For $\ell\in\mathbb{N}_0$, $\C^\ell(\Omega)$ is the space consisting of function
whose all derivatives up to the order $\ell$ belong to $\C(\Omega)$, the set of all bounded continuous functions in $\Omega$.
$\C^\infty_0(\Omega)$ is the space of infinitely differentiable function with compact support in $\Omega$.
For $1\leq p\leq \infty$ and $\ell\in\mathbb{N}$, we denote by $\L^p(\Omega)$ and $\W^{\ell,p}(\Omega)$ the Lebesgue and Sobolev function spaces.
$\W^{\ell,p}_0(\Omega)$ is defined as the closure of $\C^\infty_0(\Omega)$ in the $\W^{\ell,p}(\Omega)$--norm.
$\L^{p'}(\Omega)$ and $\W^{-\ell,p'}(\Omega)$, where $p'$ is the H\"older conjugate of $p:\frac{1}{p}+\frac{1}{p'}=1$, are (topologically and algebraically isomorphic to) the dual spaces of $\L^p(\Omega)$ and $\W^{\ell,p}_0(\Omega)$.
The duality paring between $\W^{-\ell,p'}(\Omega)$ and $\W^{\ell,p}_0(\Omega)$ is denoted by $\langle\cdot,\cdot\rangle_{\W^{-\ell,p'}(\Omega)\times\W^{\ell,p}_0(\Omega)}$, or just $\langle\cdot,\cdot\rangle$ if the duality pairing of functions spaces is implied.
The norms in $\L^p(\Omega)$, $\W^{\ell,p}(\Omega)$ and $\W^{-\ell,p'}(\Omega)$ are denoted by $\|\cdot\|_p$, $\|\cdot\|_{\ell,p}$ and $\|\cdot\|_{-\ell,p'}$.
When $p=2$, $\L^2(\Omega)$ and $\W^{\ell,2}(\Omega)$ are Hilbert spaces, with the norms induced by the scalar products $\left(\cdot,\cdot\right)$ and $\left(\left(\cdot,\cdot\right)\right)_\ell$.
In this case, the Sobolev spaces are denoted by $\H^{\ell}(\Omega):=\W^{\ell,2}(\Omega)$, $\H^{\ell}_0(\Omega):=\W^{\ell,2}_0(\Omega)$ and $\H^{-\ell}(\Omega):=\W^{-\ell,2}(\Omega)$.
The vector-valued analogues are denoted by $\CC^{0,\alpha}(\overline{\Omega})$, $\mathbf{C}^\ell(\Omega)$, $\mathbf{C}^\infty_0(\Omega)$, $\LL^p(\Omega)$, $\WW^{\ell,p}(\Omega)$, $\WW^{\ell,p}_0(\Omega)$, $\HH^{\ell}(\Omega)$, $\HH^{\ell}_0(\Omega)$, $\WW^{-\ell,p'}(\Omega)$ and $\HH^{-\ell}(\Omega)$.

\section{The governing equations}\label{Sect:GE}
Let us consider a mixture formed by $n$ constituents of Newtonian
incompressible fluids. For $i\in\{1,\dots,n\}$, we denote by $\mathbf{v}_i$
and $\rho_i$ the velocity field and density of each fluid
constituent. The incompressibility constraint and the conservation
of mass for each constituent imply
\begin{alignat}{2}
\label{eq.1.01} &\mathrm{div}\mathbf{v}_i=0\,,
\\
\label{eq.1.02} &\partial_t\rho_i+\mathrm{div}(\rho_i\mathbf{v}_i)=m_i\,,
\end{alignat}
where $m_i$ accounts for the mass supply of each constituent. The
conservation of linear momentum for each constituent requires that
\begin{equation}
\label{eq.1.03}
\partial_t\big(\rho_i\mathbf{v}_i\big)+\mathbf{div}\big(\rho_i\mathbf{v}_i\otimes\mathbf{v}_i\big)=\rho_i\mathbf{f}
_i+\mathbf{div}\boldsymbol{\sigma}_i+\mm_i+m_i\mathbf{v}_i\,,
\end{equation}
where $\mathbf{f} _i$ is the density of the body forces acting from the
external environment on the $i$-th constituent of the mixture,
$\boldsymbol{\sigma}_i$ is the Cauchy stress tensor associated with the
$i$-th constituent, $\mm_i$ is the momentum supply due to the
interaction between the $i$-th constituent and other constituents
(interactive body forces), and $m_i\mathbf{v}_i$ is the contribution to the
momentum due to the mass production $m_i$ of the $i$-th constituent.
We define the total mass density of the mixture
\begin{equation}\label{eq.1.04}
\rho:=\sum\limits_{i=1}^{n}\rho_i,
\end{equation}
the mean velocity of the mixture
\begin{equation}\label{eq.1.05}
\mathbf{v}:=\frac{1}{\rho}\sum\limits_{i=1}^{n}\rho_i\mathbf{v}_i\,,
\end{equation}
and the total stress
\begin{equation}\label{eq.1.06}
\Sig:=\sum\limits_{i=1}^{n}\Sig_i\,.
\end{equation}
The conservation of mass and linear momentum for the whole mixture
read as
\begin{alignat}{2}
\label{eq.1.07} &\partial_t\rho+\mathrm{div}(\rho\mathbf{v})=0\,,
\\
\label{eq.1.08} &
\rho\left(\partial_t\mathbf{v}+(\mathbf{v}\cdot\nabla)\mathbf{v}\right)= \rho\mathbf{f}  +\mathrm{div}\bT\,,
\end{alignat}
where $\mathbf{f} $ is the resultant of all body forces acting on the
mixture, and $\bT$ is the Cauchy stress tensor of the mixture.
Equation \eqref{eq.1.08} is written in the non-conservative form for
later use. By summing \eqref{eq.1.03} from $i=1$ till $i=n$, using
\eqref{eq.1.04}--\eqref{eq.1.05}, and comparing the resulting
equation with \eqref{eq.1.07}, it turns out that
\begin{equation*}
\sum\limits_{i=1}^{n}m_i=0\,,
  \end{equation*}
which states that the total mass of the mixture is conserved. On the
other hand, summing the non-conservative form of \eqref{eq.1.03}
from $i=1$ to $i=n$, using \eqref{eq.1.02} and
\eqref{eq.1.05}--\eqref{eq.1.06}, we obtain
\begin{equation*}
\rho\left(\partial_t\mathbf{v}+(\mathbf{v}\cdot\nabla)\mathbf{v}\right)=\sum\limits_{i=1}^{n}\rho_i\mathbf{f}
_i +
\mathbf{div}\!\left(\Sig-\sum\limits_{i=1}^{n}\rho_i\mathbf{v}_i\otimes\mathbf{v}_i\right)+\sum\limits_{i=1}^{n}\left(\mm_i+m_i\mathbf{v}_i\right)\,.
\end{equation*}
See \cite[Sections 2.2-4]{Raja-Tao:1995} for more details. By Newton's
third law, it should be
\begin{equation}
\label{eq.1.09}
\sum\limits_{i=1}^{n}\left(\mm_i+m_i\mathbf{v}_i\right)=\0\,,
\end{equation}
which means that the net momentum supply to the mixture due to all
the constituents is zero. Hence,
\begin{equation}
\label{eq.1.10}
\rho\big(\partial_t\mathbf{v}+(\mathbf{v}\cdot\nabla)\mathbf{v}\big)=\sum\limits_{i=1}^{n}\rho_i\mathbf{f}
_i +
\mathbf{div}\!\left(\Sig-\sum\limits_{i=1}^{n}\rho_i\mathbf{v}_i\otimes\mathbf{v}_i\right)\,.
\end{equation}
Comparing \eqref{eq.1.10}  with \eqref{eq.1.08}, it must be
\begin{alignat*}{2}
& \mathbf{f} =\frac{1}{\rho}\sum\limits_{i=1}^{n}\rho_i\mathbf{f} _i\,, \\ 
& \bT=\Sig-\sum\limits_{i=1}^{n}\rho_i\mathbf{v}_i\otimes\mathbf{v}_i\,.
\end{alignat*}
If we assume an isothermal process and suppose that chemical
reactions are absent, then there is no interconversion of mass
between the constituents, and therefore
\begin{equation}
\label{eq.1.11}  m_i=0\qquad \forall\ i\in\{1,\dots,n\}\,.
\end{equation}
Now, combining \eqref{eq.1.11}  with \eqref{eq.1.09} , we arrive at
\begin{equation}
\label{eq.1.12} \sum\limits_{i=1}^{n}\mm_i=\0\,.
\end{equation}

\
Darcy's law is among the most commonly used laws to study the
interactive body forces among the mixture constituents. As the
velocity of the $i$-th constituent relative to the $j$-th
constituent is defined by $\mathbf{v}_j-\mathbf{v}_i$ (slip
velocity), we can model the diffusion of one constituent through
another by using slip velocities~\cite{EZS:2024}. Taking this into
account, a general expression for the $i$-th momentum supply can be
given by
\begin{equation}\label{eq.1.13}
\mm_i=\sum\limits_{j=1}^{n}\gamma(\rho_i,\rho_j,|\mathbf{v}_j-\mathbf{v}_i|)\left(\mathbf{v}_j-\mathbf{v}_i\right),
\end{equation}
for some functions
$\gamma_{ij}=\gamma(\rho_i,\rho_j,|\mathbf{v}_j-\mathbf{v}_i|)$, that can be
arranged in a matrix
\begin{equation*}
\boldsymbol{\gamma}= \left[
\begin{array}{cccc}
0&\gamma(\rho_1,\rho_2,|\mathbf{v}_1-\mathbf{v}_2|)&\cdots&\gamma(\rho_1,\rho_n,|\mathbf{v}_1-\mathbf{v}_n|)
\\[10pt]
\gamma(\rho_2,\rho_1,|\mathbf{v}_2-\mathbf{v}_1|)&0&
\cdots&\gamma(\rho_2,\rho_n,|\mathbf{v}_2-\mathbf{v}_n|)
\\[10pt]
\vdots&\vdots&\ddots&\vdots
\\[10pt]
\gamma(\rho_n,\rho_1,|\mathbf{v}_n-\mathbf{v}_1|)
&\gamma(\rho_n,\rho_2,|\mathbf{v}_n-\mathbf{v}_2|) & \cdots  &   0
\end{array}
\right]
\end{equation*}
In this work, we assume a simplified version of \eqref{eq.1.13},
$$
\gamma(\rho_i,\rho_j,|\mathbf{v}_j-\mathbf{v}_i|)=\gamma_{ij}=\mbox{Constant}\geq0.
$$
This yields the following relation for the $i$-th momentum supply
\begin{equation}\label{eq.1.14}
\mm_i=\sum\limits_{j=1}^{n}\gamma_{ij}\left(\mathbf{v}_j-\mathbf{v}_i\right).
\end{equation}
Taking into account \eqref{eq.1.12}, \eqref{eq.1.13} and
\eqref{eq.1.14}, the following assumptions make sense
\begin{equation}\label{eq.1.15}
\gamma_{ij}=\gamma_{ji}\quad\forall\,i,j\in\{1,\dots,n\}\quad \text{
and }\quad \sum\limits_{i,j=1}^n|\gamma_{ij}|\leq \gamma^+,
\end{equation}
for some positive constant $\gamma^+$.

Regarding the constitutive equation for the Cauchy stress tensor, we
first note the conservation of angular momentum for the mixture
requires that $\Sig^T=\Sig$, which, in view of \eqref{eq.1.06},
implies
\begin{equation}\label{eq.1.16}
\sum\limits_{i=1}^{n}\left(\Sig_i-\Sig_i^T\right)=\0,
\end{equation}
On the other hand, the conservation of angular momentum for the
$i$-th constituent leads to
\begin{equation}\label{eq.1.17}
\bM_i=\Sig_i-\Sig_i^T\,,
\end{equation}
where $\bM_i$ is the angular momentum supply to the $i$-th
constituent. It should be noted that although the total stress
tensor $\Sig$ of the mixture must be symmetric, the partial stresses
$\Sig_i$ need not be. However, if we assume body couples are absent,
then $\bM_i=\0$ and, in view of \eqref{eq.1.17}, $\Sig_i$ is
symmetric~\cite{Raja-Tao:1995}. In this case, \eqref{eq.1.16} is
immediately  verified.

We split the Cauchy stress tensor of each constituent into spherical
and deviatoric parts,
\begin{equation}\label{eq.1.18}
\Sig_i=-p_i\Ii+\Tau_i\,,
\end{equation}
where $p_i$ is the pressure acting on the $i$-th fluid constituent,
$\Ii$ is the identity tensor and $\Tau_i$ is the partial deviatoric
tensor. Note that in \eqref{eq.1.18} $-p_i\Ii$ accounts for the
reaction stress due to the constraint of incompressibility condition
\eqref{eq.1.01}. We assume the deviatoric tensor results from
combining Hook's law of linear elasticity with Newton's law of
viscosity as follows
\begin{equation}\label{eq.1.19}
\Tau_i=2\mu_i\Du(\mathbf{v}_i)+2\kappa_i\partial_t\Du(\mathbf{v}_i)\,,\qquad
\Du(\mathbf{v}_i):=\frac{1}{2}\left(\nabla\mathbf{v}_i+\nabla\mathbf{v}_i^T\right)\,,
\end{equation}
where $\mu_i$ and $\kappa_i$ are viscosity and elasticity (positive)
constants specific to each fluid constituent, $\Du(\mathbf{v}_i)$ is the
symmetric velocity gradient of $i$-th constituent (rate of strain
tensor). Equation \eqref{eq.1.19} is a generalization to higher
dimensions of the one-dimensional Kelvin-Voigt model that has been
used in the applications to model the rheology of viscous fluids
with elastic properties, as polymers and certain human
tissues~\cite{Pav:1971,CSS:2002}. The classical one-dimensional
Kelvin-Voigt model
\begin{equation}\label{eq.1.20}
\tau=\mu\frac{\partial v}{\partial
x}+\kappa\frac{\partial^2v}{\partial x\partial t}
\end{equation}
gives a simple constitutive relation describing the behavior of
linear viscoelastic materials, where $\mu$ denotes the fluid
viscosity and $\kappa$ accounts for the relaxation time, that is the
time taken by the fluid to return from the deformed state to its
initial equilibrium state. Constitutive equation \eqref{eq.1.20} can
also be justified by using the Boltzmann superposition principle, or
by considering mechanical models constructed using elastic springs
and viscous dashpots~\cite{Barnes:2000}. Constitutive equations that
seem to have any theoretical basis and practical evidence in the
mixture of fluids are those involving mutual viscosity and
elasticity constants~\cite{SW:1977,Massoudi:2008}. A law that takes
this into account and generalizes \eqref{eq.1.19}, can be written as
follows,
\begin{equation}\label{eq.1.21}
\Tau_i=\sum\limits_{j=1}^{n}\big(2\mu_{ij}\Du(\mathbf{v}_j)+2\kappa_{ij}\partial_t\Du(\mathbf{v}_j)\big)
\,,
\end{equation}
where
$\mu_{ij}$ and $\kappa_{ij}$ are mutual viscosity and elasticity
constants, that is $\mu_{ij}$ and $\kappa_{ij}$ measure viscous and
elastic forces exerted by molecules of constituent $i$ and molecules
of constituent $j$. When $j=i$, $\mu_{ii}$ and $\kappa_{ii}$ measure
forces exerted between molecules of the same fluid constituent, as
in the flow of a single fluid.

In fluid mechanics experiments, the choice of constitutive equation
depends to a considerable extent on the use to which it is to be
put. For the characterization of many viscoelastic fluids, equations
of the form \eqref{eq.1.18}--\eqref{eq.1.21} are generally accepted
as the correct starting point for the description of the rheology of
incompressible viscoelastic fluids~\cite{BAH:1987a}. As there is a
large number of possible relations that one could propose linking
stresses and strains, some admissibility conditions are usually
imposed to narrow the field of possible relations. For
incompressible fluids, Oldroyd~\cite{Oldroyd:1950} proposed to
restrict consideration to models that are: (a) frame invariant; (b)
deterministic; (c) local. However, the elastic tensor
$\partial_t\Du(\mathbf{v}_i)$ is not objective, and therefore does not satisfy
the principle of frame indifference. To see this, let
$\Du(\mathbf{v}_i)^\ast$ and $\partial_t\Du(\mathbf{v}_i)^\ast$ be the representations
of $\Du(\mathbf{v}_i)$ and $\partial_t\Du(\mathbf{v}_i)$ in a second reference frame,
and consider an arbitrary orthogonal tensor $\Qq=\Qq(\cdot,t)$. As
$\Du(\mathbf{v}_i)$ is objective, and so
$\Du(\mathbf{v}_i)^\ast=\Qq\Du(\mathbf{v}_i)^\ast\Qq^T$, we should have
\begin{equation*}
\partial_t\Du(\mathbf{v}_i)^\ast=\partial_t\Qq\Du(\mathbf{v}_i)\Qq^T+\Qq\partial_t\Du(\mathbf{v}_i)^\ast
\Qq^T+\Qq\Du(\mathbf{v}_i)\partial_t\Qq^T.
\end{equation*}
Hence, the objectivity of $\partial_t\Du(\mathbf{v}_i)$ would require
\begin{equation*}
\partial_t\Qq\Du(\mathbf{v}_i)\Qq^T+\Qq\Du(\mathbf{v}_i)\partial_t\Qq^T=\0\,,
\end{equation*}
which is not always guaranteed. To construct an admissible
constitutive equation, we can replace the partial time-derivative in
\eqref{eq.1.19} by the following convected
time-derivative~\cite{BAH:1987a},
\begin{equation}\label{eq.1.22}
\Du(\mathbf{v}_i)^\triangledown=\partial_t\Du(\mathbf{v}_i)+(\mathbf{v}_i\cdot\nabla)\Du(\mathbf{v}_i)-\Big(\Du(\mathbf{v}_i)\nabla\mathbf{v}_i+(\Du(\mathbf{v}_i)\nabla\mathbf{v}_i)^T\Big)\,.
\end{equation}
However, PDEs resulting from constitutive equations involving the
convected time-derivative \eqref{eq.1.22} become too complex to
study, so some simplification is necessary for the purposes of their
mathematical analysis.
The Cauchy stress tensor in a viscoelastic fluid has a contribution
generated by the viscous and elastic parts of the fluid. However,
the elastic part of the fluid does not, by definition, dissipate the
energy. This implies that the product $\Tau_i:\Du(\mathbf{v}_i)$ does not
provide a characterization of the dissipation in the material.
Indeed, the stress $\Tau_i$ in \eqref{eq.1.21} necessarily contains
a non-dissipative part that does not vanish in the product
$\Tau_i:\Du(\mathbf{v}_i)$. In particular, the term $\Tau_i:\Du(\mathbf{v}_i)$ is
not always positive as in the standard incompressible Navier-Stokes
fluid where the positivity of the viscosities $\mu_{ij}$ imply
$\Tau_i:\Du(\mathbf{v}_i)\geq0$
(see~\cite{HMPST:2017}). In this work, we assume a generalization of
the hypothesis  underlying the one dimensional model, that is,
$\mu_{ij}$ and $\kappa_{ij}$ are assumed to be constants such that
$\displaystyle\mathbf{M}=\{\mu_{ij}\}_{i, j = 1}^{n}$
 and $\displaystyle\mathbf{K}=\{\kappa_{ij}\}_{i, j = 1}^{n}$
 are positive definite matrixes, \emph{i.e.}, there exist positive constants $\mu^-$, $\mu^+$, $\kappa^-$ and $\kappa^+$ such that
\begin{alignat}{2}
\label{eq.1.24} &0<\mu^-|\boldsymbol{\zeta}|^2
\leq(\mathbf{M}\boldsymbol{\zeta},\boldsymbol{\zeta})\leq
\mu^+|\boldsymbol{\zeta}|^2,\qquad
\forall\,\boldsymbol{\zeta}\in\mathbb{R}^n\backslash\{\boldsymbol{0}\},
\\
\label{eq.1.25} &0<\kappa^-|\boldsymbol{\zeta}|^2\leq
(\mathbf{K} \boldsymbol{\zeta},\boldsymbol{\zeta})\leq
\kappa^+|\boldsymbol{\zeta}|^2,\qquad
\forall\,\boldsymbol{\zeta}\in\mathbb{R}^n\backslash\{\boldsymbol{0}\}.
\end{alignat}
Combining \eqref{eq.1.01}--\eqref{eq.1.03}, \eqref{eq.1.14},
\eqref{eq.1.18} and \eqref{eq.1.21}, we obtain the system of partial
differential equations \eqref{eq.1.26}--\eqref{eq.1.28}.

{Note that if $\mathbf{M}$, $\mathbf{K}$ and $\boldsymbol{\gamma}$ are
diagonal matrixes, then the subsystem of $n$ equations
\eqref{eq.1.27} decouples and, therefore, the whole system
\eqref{eq.1.26}--\eqref{eq.1.28} falls into the well-known theory of
Kelvin-Voigt equations that govern flows of one-constituent viscoelastic
fluids~\cite{AOKh:2020, AOKh:2019:b, AOKh:2019:c,
Antontsev-2021, Mohan:2020, Oskolkov:1989, Oskolkov:1985,
Turbin:2023,  Zvyagin:2024, Zvyagin:2023, Zvyagin-2018-1}. In the
present work, we consider a much more complicated case of
non-diagonal and non-triangular matrixes $\mathbf{M}$,
$\mathbf{K}$, and $\boldsymbol{\gamma}$. Solvability topics for
fluid mechanics equations with non-diagonal viscosity matrixes have been carried out by several authors, amongst whom~\cite{AC:1976a,Baris:2013,Beevers:1982,Frehse_SJMA_2005,Frehse_AM_2005,Frehse_AM_2008,Kirwan.Fluids:2020, Mamontov:2021, Mamontov:2014,Prokudin:2021}.}

\section{Existence of weak solutions}
\label{Sect:EWS} Let us introduce the following function spaces
widely used in Mathematical Fluid Mechanics:
\begin{align*}
&\mathcal{V}:=\{\mathbf{v}\in \mathbf{C}_{0}^\infty(\Omega):\,\mathrm{div}\mathbf{v}=0\},
\\
&\HH:=\text{ closure of }\mathcal{V} \text{ in the $\LL^2(\Omega)$--norm},
\\
&\VV^{\ell}:=\text{ closure of }\mathcal{V}\text{ in the $\WW^{\ell,2}(\Omega)$--norm},\quad \ell\in\mathds{N}.
\end{align*}
As usual, when $\ell=1$, we denote $\VV^{\ell}$ only by $\VV$, instead of $\VV^{1}$.

\begin{conditions}
\label{CondsA} Along with the hypotheses \eqref{eq.1.15} and
\eqref{eq.1.24}--\eqref{eq.1.25}, we assume the following conditions hold:
\begin{itemize}
\item  $\Omega\subset\mathbb{R}^d$ is a bounded domain, $d \in\{2,3\}$, with its boundary $\partial\Omega$ supposed to be Lipschitz-continuous;
\item the initial data is regular in the sense that
\begin{equation}\label{eq.2.01}
\rho_{i,0}\in \C^1(\overline{\Omega})\quad \wedge\quad \mathbf{v}_{i,0}\in
\VV \qquad \forall\ i\in\{1,\dots,n\};
\end{equation}
\item the initial densities satisfy
\begin{equation}\label{eq.2.02}
0<\rho^{-}\leq
\min\limits_{i\in\{1,\dots,n\}}\left\{\inf_{\mathbf{x}\in\overline{\Omega}}\rho_{i,0}(\mathbf{x})\right\}\leq
\rho_{i,0}(\mathbf{x})\leq
\max\limits_{i\in\{1,\dots,n\}}\left\{\sup_{\mathbf{x}\in\overline{\Omega}}\rho_{i,0}(\mathbf{x})\right\}\leq
\rho^+\quad \forall\ \mathbf{x}\in\overline{\Omega},
\end{equation}
for some positive constants $\rho^{-}$ and $\rho^+$, and all $ i\in\{1,\dots,n\}$;
\item the forces fields are integrable in the sense that
\begin{alignat}{2}
\label{eq.2.03} & \mathbf{f}_i\in \L^2(0,T;\LL^{2}(\Omega))\qquad  \forall\ i\in\{1,\dots,n\}. &&
\end{alignat}
\end{itemize}
\end{conditions}

A weak solution to the problem \eqref{eq.1.26}--\eqref{eq.1.30} is
defined as follows.
 \begin{definition}\label{def:ws}
Assume Conditions~\ref{CondsA} are verified.
The $n$-uplet pair of functions
$(\rho_1,\mathbf{v}_1),\dots, (\rho_n,\mathbf{v}_n)$ is said to be a weak solution
to the problem \eqref{eq.1.26}--\eqref{eq.1.30}, if for each
$i\in\{1,\dots,n\}$:
\begin{enumerate}
\item $\mathbf{v}_i\in \L^{\infty}(0,T;\VV)\cap \C([0,T];\HH)$, $\partial_t\mathbf{v}_i\in \L^2(0,T;\VV)$;
\item  {$\rho_i>0$ a.e. in $Q_T$}, $\rho_i\in \C([0,T];\L^{\lambda}(\Omega))$ for all $\lambda\in[1,+\infty)$,
 and $\rho_i|\mathbf{v}_i|^2\in \L^{\infty}(0,T;\L^{1}(\Omega))$;
\item {$\rho_i(\mathbf{x},0)=\rho_{i,0}(\mathbf{x})$ and $\rho_i(\mathbf{x},0)\mathbf{v}_i(\mathbf{x},0)=\rho_{i,0}(\mathbf{x})\mathbf{v}_{i,0}(\mathbf{x})$} in $\Omega$, with $\rho_{i,0}>0$ a.e. in $\Omega$;
\item For every $\boldsymbol{\phi}_i\in\mathcal{V}$ there holds
\begin{equation}
\label{eq.2.04}%
\begin{split}
\frac{d}{dt}&\left(\int_{\Omega}\rho_i(\mathbf{x},t)\mathbf{v}_i(\mathbf{x},t)\cdot\boldsymbol{\phi}_i\,d\mathbf{x}
+\sum\limits_{j=1}^{n}\kappa_{ij}\int_\Omega\nabla\mathbf{v}_j(\mathbf{x},t)
:\nabla\boldsymbol{\phi}_i\,d\mathbf{x} \right)
\\&
+\sum\limits_{j=1}^{n}\mu_{ij}\int_\Omega\nabla\mathbf{v}_j(\mathbf{x},t):\nabla\boldsymbol{\phi}_i\,d\mathbf{x}
=\int_\Omega \rho_i(\mathbf{x},t)\mathbf{f}_i(\mathbf{x},t)\cdot\boldsymbol{\phi}_i\,d\mathbf{x}
\\&
+\int_\Omega
\rho_i(\mathbf{x},t)\mathbf{v}_i(\mathbf{x},t)\otimes\mathbf{v}_i(\mathbf{x},t):\nabla\boldsymbol{\phi}_i\,d\mathbf{x}
+\sum\limits_{j=1}^{n}\gamma_{ij}\int_{\Omega}\left(\mathbf{v}_j(\mathbf{x},t)-\mathbf{v}_i(\mathbf{x},t)\right)\cdot\boldsymbol{\phi}_i\,d\mathbf{x},
\end{split}
\end{equation}
in the distribution sense on $(0,T)$.
\item For every $\phi_i\in\CC^{\infty}_{0}(\Omega)$
there holds
\begin{equation}
\label{eq.2.05} \frac{d}{dt}\int_{\Omega}\rho_i(\mathbf{x},t)\,\phi_i\,d\mathbf{x} -
\int_{\Omega}\rho_i(\mathbf{x},t)\mathbf{v}_i(\mathbf{x},t)\cdot{\nabla} \phi_i\,d\mathbf{x} =0,
\end{equation}
in the distribution sense on $(0,T)$.
\end{enumerate}
\end{definition}

In the following theorem we establish the main existence result of this work.

\begin{theorem}[\textbf{Global-in-time existence}] \label{Theorem.2.1}
Let Conditions~\ref{CondsA} be verified.
If, in addition, conditions \eqref{eq.1.15} and \eqref{eq.1.24}--\eqref{eq.1.25} hold, then the problem
\eqref{eq.1.26}--\eqref{eq.1.30} has a unique global-in-time weak
solution $\{(\rho_1,\mathbf{v}_1),\dots,(\rho_n,\mathbf{v}_n)\}$ in the sense of
Definition \ref{def:ws}. In addition, there exist positive constants $K_1$ and $K_2$
\begin{align}
\label{eq.2.06}
&\sum\limits_{i=1}^{n}\left(\esssup_{t\in[0,T]}\|\mathbf{v}_i(\cdot,t)\|_{2}^2+
\esssup_{t\in[0,T]}\|\nabla \mathbf{v}_i(\cdot,t)\|_{2}^2\right)\leq K_1,
\\
\label{eq.2.07}
&\sum\limits_{i=1}^{n}\int_0^T\left(\|\sqrt{\rho_i(\cdot,t)}\partial_t\mathbf{v}_i(\cdot,t)\|_{2}^2+
\|\nabla\partial_t\mathbf{v}_i(\cdot,t)\|_{2}^2\right)dt\leq K_2.
\end{align}
\end{theorem}
Constants $K_1$ and $K_2$, considered in \eqref{eq.2.06} and \eqref{eq.2.07}, have the same dependence on the problem data as the right-hand side positive constants of the estimates \eqref{eq.2.28} and \eqref{eq.2.36} established in Lemmas~\ref{Lem.2.2} and \ref{Lem.2.3} below.

For the proof of Theorem~\ref{Theorem.2.1}, we shall use a
semi-implicit Galerkin approximation method. To this end, it is
important to recall the following classical result.
\begin{lemma}[\cite{LS:1978,Simon:1990}]\label{lem:LS}
Let $\mathbf{v}\in \C^1([0,T];\CC^1(\overline{\Omega}))$ and assume
$\mathbf{v}=\boldsymbol{0}$ on $\partial\Omega$. If $\rho_0\in
\C^1(\overline{\Omega})$ and
$0<\rho^{-}\leq\rho_{0}(\mathbf{x})\leq\rho^+<\infty$ for all
$\mathbf{x}\in\overline{\Omega}$, then the following initial-value problem
\begin{alignat*}{2}
& \partial_t\rho+\mathrm{div} (\rho\mathbf{v})=0,\quad \mathrm{div}\mathbf{v}=0,\quad \rho>0,\quad  \mbox{in}\quad Q_T, \\
& \rho=\rho_0 \qquad \mbox{in}\quad \Omega\times\{0\}
\end{alignat*}
has a unique (classical) solution $\rho\in \C^1([0,T];\C^1(\overline{\Omega}))$. Moreover,
\begin{equation}\label{eq.2.08}
0<\rho^-\leq\rho(\mathbf{x},t)\leq\rho^+\qquad \forall\ (\mathbf{x},t)\in Q_T.
\end{equation}
\end{lemma}

\begin{proof}(Theorem~\ref{Theorem.2.1})
For the sake of comprehension, this proof will be split  into several subsections.

\subsection{\bf Galerkin approximations}
\label{Sec.2.1} We construct a solution to the problem
\eqref{eq.1.26}--\eqref{eq.1.30} as a limit of suitable Galerkin
approximations. One-constituent case was treated in
\cite{Antontsev-2021,deOliveira}. It is well-known the existence of
an increasing sequence of positive eigenvalues $\lambda_{k}$ and a
sequence of corresponding eigenfunctions
\begin{equation}\label{eq.2.09}
\boldsymbol{\psi}_{k}\in \CC^1(\overline{\Omega})\cap
\VV
\end{equation}
to the following problem
\begin{alignat}{2}
\label{eq.2.10}
&A\boldsymbol\psi_{k}=\lambda_k\boldsymbol\psi_{k}\quad\mbox{in}\ \
\Omega, &&
\\
\label{eq.2.11} &\boldsymbol\psi_k=\boldsymbol{0}\quad\mbox{on}\ \
\partial\Omega,
\end{alignat}
for the Stokes operator $A=-P\Delta$, where
$P:\LL^{2}(\Omega)\mapsto\HH$ is the Leray-Helmholtz operator. Since the Stokes operator is
injective, self-adjoint and has a compact inverse, the family
$\left\{\boldsymbol\psi_k\right\}_{k\in\mathbb{N}}$ can be made
orthogonal in $\HH$ and orthonormal in
$\VV$~\cite[Propositions~4.2-4]{CF:1988}. Given
$m\in\mathbb{N}$, let us consider the $m$-dimensional space
$\mathbb{X}^{m}$ spanned by the first $m$ eigenvalues satisfying
\eqref{eq.2.10}--\eqref{eq.2.11},
$\boldsymbol\psi_1,\ldots,\boldsymbol\psi_m$. For each
$m\in\mathbb{N}$, and some $T>0$, we search for approximate
solutions
\begin{alignat}{2}
\label{eq.2.12} & \mathbf{v}_1^{(m)},\dots,\mathbf{v}_n^{(m)}\in
\C^1([0,T];\mathbb{X}^{m}),
\\
\label{eq.2.13} & \rho_1^{(m)},\dots,\rho_n^{(m)}\in
\C^1([0,T];\C^1(\overline{\Omega}))
\end{alignat}
to the problem \eqref{eq.1.26}--\eqref{eq.1.30} in the form
\begin{equation}
\label{eq.2.14}
\mathbf{v}_i^{(m)}(\mathbf{x} ,t)= \sum\limits_{k=1}^{m}c_{i,k}^{m}(t)
\boldsymbol{\psi}_k(\mathbf{x}),\qquad\boldsymbol{\psi}_k\in\mathbb{X}^{m},
\end{equation}
where the coefficient functions $c_{i,1}^m,\ldots,c_{i,m}^m \in
\C^1([0,T])$ are defined as the solutions of the following $n\times
m$ ordinary differential equations derived from \eqref{eq.2.04},
\begin{equation}
\label{eq.2.15}%
\begin{split}
&
\int_{\Omega}\partial_t\big(\rho_i^{(m)}(\mathbf{x},t)\mathbf{v}_i^{(m)}(\mathbf{x},t)\big)\cdot
\boldsymbol{\psi}_k\,d\mathbf{x} +
\sum\limits_{j=1}^{n}\kappa_{ij}\int_\Omega
\nabla\partial_t\mathbf{v}_{j}^{(m)}(\mathbf{x},t):\nabla\boldsymbol{\psi}_k\,d\mathbf{x}
\\
&+
\sum\limits_{j=1}^{n}\mu_{ij}\int_\Omega\nabla\mathbf{v}_j^{(m)}(\mathbf{x},t):\nabla\boldsymbol\psi_k\,d\mathbf{x}
=\int_\Omega
\rho_i^{(m)}(\mathbf{x},t)\mathbf{f}_i(\mathbf{x},t)\cdot\boldsymbol\psi_k\,d\mathbf{x}
\\
& + \int_\Omega
\rho_i^{(m)}(\mathbf{x},t)\mathbf{v}_i^{(m)}(\mathbf{x},t)\otimes\mathbf{v}_i^{(m)}(\mathbf{x},t):\nabla\boldsymbol\psi_k\,d\mathbf{x}
\\
&+\sum\limits_{j=1}^{n}\gamma_{ij}\int_{\Omega}\left(\mathbf{v}_j^{(m)}(\mathbf{x},t)-\mathbf{v}_i^{(m)}(\mathbf{x},t)\right)\cdot\boldsymbol\psi_k\,d\mathbf{x},
\end{split}
\end{equation}
for $k\in\{1,\dots,m\}$, $i\in\{1,\dots,n\}$. Moreover,
$\rho_i^{(m)}$ satisfies the following equation in the classical sense
\begin{equation}\label{eq.2.16}
\partial_t\rho_i^{(m)}+\mathrm{div}\big(\rho_i^{(m)}\mathbf{v}_i^{(m)}\big)=0.
\end{equation}
We recall that \eqref{eq.2.15} can be written in the nonconservative
form
\begin{equation}\tag{\ref{eq.2.15}${}^\prime$}\label{eq.2.15'}
\begin{split}
&\int_{\Omega}\rho_i^{(m)}(\mathbf{x},t)\partial_t\mathbf{v}_i^{(m)}(\mathbf{x},t)\cdot\boldsymbol{\psi}_k\,d\mathbf{x}
+
\sum\limits_{j=1}^{n}\kappa_{ij}\int_\Omega\nabla\partial_t\mathbf{v}_{j}^{(m)}(\mathbf{x},t):\nabla\boldsymbol{\psi}_k\,d\mathbf{x}
\\
&
+\sum\limits_{j=1}^{n}\mu_{ij}\int_\Omega\nabla\mathbf{v}_j^{(m)}(\mathbf{x},t):\nabla\boldsymbol\psi_k\,d\mathbf{x}
\\
&
+\int_\Omega\rho_i^{(m)}(\mathbf{x},t)\mathbf{div}\left(\mathbf{v}_i^{(m)}(\mathbf{x},t)\otimes\mathbf{v}_i^{(m)}(\mathbf{x},t)\right)\cdot\boldsymbol{\psi}_k\,d\mathbf{x}
=
\\
& \int_\Omega
\rho_i(\mathbf{x},t)\mathbf{f}_i(\mathbf{x},t)\cdot\boldsymbol\psi_k\,d\mathbf{x}+
\sum\limits_{j=1}^{n}\gamma_{ij}\int_{\Omega}\left(\mathbf{v}_j(\mathbf{x},t)-\mathbf{v}_i(\mathbf{x},t)\right)\cdot\boldsymbol\psi_k\,d\mathbf{x},
\end{split}
\end{equation}
for $k\in\{1,\dots,m\}$ and $i\in\{1,\dots,n\}$.

System \eqref{eq.2.15'}--\eqref{eq.2.16} is supplemented with the
following initial conditions
\begin{equation}\label{eq.2.17}
\mathbf{v}_i^{(m)}(\mathbf{x},0)=\mathbf{v}_{i,0}^{(m)}(\mathbf{x}),\quad
\rho_i^{(m)}(\mathbf{x},0)=\rho_{i,0}^{(m)}(\mathbf{x})\quad \mbox{in}\ \ \Omega,\qquad
i\in\{1,\dots,n\},
\end{equation}
where, for each $i\in\{1,\dots,n\}$,
$\mathbf{v}_{i,0}^{(m)}:=P^m(\mathbf{v}_{i,0})$, and $P^m$ denotes the orthogonal projection
$P^m:\VV\longrightarrow\mathbb{X}^m$, which is defined by
\begin{equation*}
P^m(\mathbf{v})=\sum\limits_{k=1}^{m}\left(\mathbf{v},\boldsymbol\psi_k\right)\boldsymbol\psi_k\,.
\end{equation*}
Since the operator $P^m$ is uniformly continuous, we can
assume that
\begin{equation}\label{eq.2.18}
\mathbf{v}_{i,0}^{(m)}\xrightarrow[m\to+\infty]{}\mathbf{v}_{i,0}\quad \mbox{in}\
\ \VV,\qquad i\in\{1,\dots,n\}.
\end{equation}

Thus, from \eqref{eq.2.18}, one easily has
\begin{alignat}{2} &
\|\mathbf{v}_{i,0}^{(m)}\|_{2}\leq \|\mathbf{v}_{i,0}\|_{2},\qquad
\|\nabla\mathbf{v}_{i,0}^{(m)}\|_{2}\leq
\|\nabla\mathbf{v}_{i,0}\|_{2},\qquad i\in\{1,\dots,n\}. &&
\label{eq.2.19}
\end{alignat}
Note that, since
$\mathbf{v}_{1,0}^{(m)},\dots,\mathbf{v}_{n,0}^{(m)}\in\mathbb{X}^m$, for each
$i\in\{1,\dots,n\}$ there exist
$c_{i,1}^{0},\ldots,c_{i,m}^{0}\in\mathbb{R}$ such that
\begin{equation*}
\mathbf{v}_{i,0}^{(m)}(\mathbf{x})=\sum\limits_{k=1}^mc_{i,k}^{0}
\boldsymbol{\psi}_k(\mathbf{x}),
\end{equation*}
and, from \eqref{eq.2.14},
\begin{equation*}
\mathbf{v}_i^{(m)}(\mathbf{x},0)=
\sum\limits_{k=1}^{m}c_{i,k}(0)\boldsymbol{\psi}_k(\mathbf{x}).
\end{equation*}
Hence, from \eqref{eq.2.17}$_1$, we deduce that
\begin{equation*}
c_{i,k}(0)=c_{i,k}^{0},\quad k\in\{1,\dots,m\},\quad
i\in\{1,\dots,n\}.
\end{equation*}
With respect to the initial approximate densities
$\rho_{1,0}^{(m)},\dots,\rho_{n,0}^{(m)}$, we assume that, for each
$i\in\{1,\dots,n\}$,
\begin{alignat}{2}
\label{eq.2.20} & \rho_{i,0}^{(m)}\in \C^1(\overline{\Omega}),
\\
\label{eq.2.21}&
\rho_{i,0}^{(m)}\xrightarrow[m\to+\infty]{}\rho_{i,0}\quad
\mbox{in}\ \
\L^p(\Omega)\quad \forall\,p\in[1,+\infty), \\
\label{eq.2.22} &
0<\rho^{-}\leq\rho_{i,0}(\mathbf{x})\leq\rho^+<\infty\quad \forall\,
\mathbf{x}\in\overline{\Omega}.
\end{alignat}

For $i\in\{1,\dots,n\}$, and in view of \eqref{eq.2.12},
\eqref{eq.2.14}, \eqref{eq.2.20}--\eqref{eq.2.22},
Lemma~\ref{lem:LS} associates each $\mathbf{v}_i^{(m)}$ with a unique
density $\rho_i^{(m)}\in \C^1([0,T_m];\C^1(\overline{\Omega}))$, cf.
\eqref{eq.2.13}, which then satisfies \eqref{eq.2.16} in the
classical sense. Moreover, from \eqref{eq.2.08} and \eqref{eq.2.09}
\begin{equation}
\label{eq.2.23}
0<\rho^-\leq\rho_i^{(m)}(\mathbf{x},t)\leq\rho^+\quad \forall\, (\mathbf{x},t)\in
Q_T,
\end{equation}
for any $i\in\{1,\dots,n\}$.

Fixing $\rho_i^{(m)}$, we construct $\mathbf{v}_i^{(m)}$ in the form
\eqref{eq.2.14}, by solving the Galerkin system
\eqref{eq.2.15}{}$^\prime$. This in turn can be written in the
following matrix form
\begin{equation}\label{eq.2.24}
\sum\limits_{j=1}^{n}\mathbf{A}_{ij}(t)\dot{\mathbf{c}}_j(t)
=\mathbf{b}_i(t), \quad\mathbf{c}_i(0)=\mathbf{c}_{i,0},\qquad
i\in\{1,\dots,n\},
\end{equation}
where, for each $i\in\{1,\dots,n\}$,
\begin{equation*}
\mathbf{A}_{ij}(t)=\left\{a_{ij,kl}^m(t)\right\}_{k,l=1}^m,\
\mathbf{c}_i=\{c_{i,1}^m,\ldots,c_{i,m}^m\},\
\mathbf{c}_{i,0}=\{c_{i,1}^{0},\ldots,c_{i,m}^{0}\},\
\mathbf{b}_i=\{b_{i,1},\ldots,b_{i,m}\},
\end{equation*}
with
\begin{align*}
& a_{ii,kl}^m(t):=
\int_{\Omega}\rho_i^{(m)}\boldsymbol{\psi}_k\cdot\boldsymbol{\psi}_{l}\,d\mathbf{x}
+\kappa_{ii}\int_\Omega\nabla\boldsymbol{\psi}_k:\nabla\boldsymbol{\psi}_{l}\,d\mathbf{x},
\\
&
a_{ij,kl}^m(t)=a_{ji,kl}^m(t):=\kappa_{ij}\int_\Omega\nabla\boldsymbol{\psi}_k:
\nabla\boldsymbol{\psi}_{l}\,d\mathbf{x},\quad i\neq j,
\\
& b_{i,k}^m(t):=\int_\Omega
\rho_i^{(m)}\mathbf{f}_i\cdot\boldsymbol\psi_k\,d\mathbf{x}
+\sum\limits_{j=1}^{n}\gamma_{ij}\int_{\Omega}\left(\mathbf{v}_j^{(m)}
-\mathbf{v}_i^{(m)}\right)\cdot\boldsymbol\psi_k\,d\mathbf{x}
\\
&
-\int_\Omega\rho_i^{(m)}\mathbf{div}(\mathbf{v}_i^{(m)}\otimes\mathbf{v}_i^{(m)})\cdot\boldsymbol{\psi}_k\,d\mathbf{x}
-\sum\limits_{j=1}^{n}\mu_{ij}\int_\Omega
\nabla\mathbf{v}_j^{(m)}:\nabla\boldsymbol\psi_k\,d\mathbf{x},
\end{align*}
where $i,j\in\{1,\ldots,n\}$. Taking into account that the family
$\left\{\boldsymbol{\psi}_{k}\right\}_{k\in\mathbb{N}}$ is linearly
independent, then, in view of \eqref{eq.2.23},
\[
\sum\limits_{i,j=1}^n\mathbf{A}_{ij}(t)\boldsymbol{\xi}_i\cdot
\boldsymbol{\xi}_j>0
\]
for all
$\boldsymbol{\xi}_i\in\mathbb{R}^m\setminus\{{\mathbf{0}}\}$,
$i\in\{1,\ldots,n\}$. Let
\begin{equation*}
\mathbf{A}=\left[\begin{array}{lll} \mathbf{A}_{11}&
\cdots&\mathbf{A}_{1n}
\\
\,\,\vdots & \ddots &\,\,\vdots
\\
\mathbf{A}_{n1}&\cdots &\mathbf{A}_{nn}
\end{array}\right],\,\,\mathbf{c}=(\mathbf{c}_1,\ldots,\mathbf{c}_n),\,\,\mathbf{c}_0=(\mathbf{c}_{1,0},\ldots,\mathbf{c}_{n,0}),\,\,\mathbf{b}=(\mathbf{b}_1,\ldots,\mathbf{b}_n).
\end{equation*}
Thus, we can write \eqref{eq.2.24} in the form
\begin{equation} \label{eq.2.25}
\frac{d\mathbf{c}}{dt}=\mathbf{A}^{-1}\mathbf{b},\quad\mathbf{c}(0)=\mathbf{c}_{0}.
\end{equation}
Setting
\begin{equation*}
\mathbf{g}(t,\mathbf{c}(t)):=\mathbf{A}^{-1}(t)\mathbf{b}(t,\mathbf{c}(t))
\end{equation*}
and considering a ball
\[
B_\varepsilon(\mathbf{c}_0)=\{\boldsymbol{\eta}\in\mathbb{R}^{nm}:|\boldsymbol{\eta}-\mathbf{c}_0|<\varepsilon\},
\]
for some $\varepsilon>0$, we can prove that $\mathbf{g}(\cdot,{\boldsymbol{\eta}})$ is
measurable for all ${\boldsymbol{\eta}}\in B_\varepsilon(\mathbf{c}_0)$, $\mathbf{g}(t,\cdot)$
is continuous for a.e. $t\in (0,T)$, $|\mathbf{g}(t,{\boldsymbol{\eta}})|\leq F(t)$ for
all $(t,{\boldsymbol{\eta}})\in(0,T)\times B_\varepsilon(\mathbf{c}_0)$ and some
integrable function $F$ in $(0,T)$. Then, by the Carath\'{e}odory
theorem, there exists $T_m\in(0,T)$ and an absolutely continuous
function $\mathbf{c}:(0,T_m)\longrightarrow\mathbb{R}^{nm}$ such that $\mathbf{c}$
solves the Cauchy problem \eqref{eq.2.25}.

By \eqref{eq.2.23},
\begin{equation}\label{eq.2.26}
0<\rho^-\leq\rho_i^{(m)}(\mathbf{x},\tau)\leq\rho^+\quad \forall\ (\mathbf{x},\tau)\in
Q_{t}:=\Omega\times[0,t].
\end{equation}
As $\mathbb{X}^{m}$ is finite dimensional, the $\L^\infty$ and $\L^2$
norms are equivalent in $\mathbb{X}^{m}$, we can combine
\eqref{eq.2.28} below and \eqref{eq.2.26}  so that
\begin{equation}\label{eq.2.27}
\sup_{t\in[0,T_m]}\left(\|\mathbf{v}_i^{(m)}(\cdot,t)\|_{\infty} +
\|\nabla \mathbf{v}_i^{(m)}(\cdot,t)\|_{\infty}\right)\leq C(m).
\end{equation}
Finally, \eqref{eq.2.26} and \eqref{eq.2.27}  enable us to use the
Continuation Principle (see~\cite[Theorem~1.4.1]{CL:1955}) to extend
the solution $\mathbf{c}$ of the Cauchy problem \eqref{eq.2.25} onto the
whole time interval $[0,T]$.

\subsection{\bf Uniform estimates for $\mathbf{v}_i^{(m)}$ and $\nabla\mathbf{v}_i^{(m)}$}
\begin{lemma}\label{Lem.2.2}
For each $i\in\{1,\ldots,n\}$, let $\mathbf{v}_i^{(m)}$ be the Galerkin
approximation considered in \eqref{eq.2.14}. If the conditions of
Theorem~\ref{Theorem.2.1} are satisfied, and, in addition,
$\kappa_{ij}=\kappa_{ji}, \ \gamma_{ij}=\gamma_{ji},$
 then there
exists a positive constant $C=C(\rho^\pm,\mu^+,\kappa^+,\gamma^+)$
such that
\begin{align}
\label{eq.2.28}
\begin{split}
&\sum\limits_{i=1}^{n}\left(\sup_{t\in[0,T]}\|\mathbf{v}_i^{(m)}(\cdot,t)\|_{2}^2+\sup_{t\in[0,T]}\|\nabla\mathbf{v}_i^{(m)}(\cdot,t)\|_{2}^2\right)
+\sum\limits_{i=1}^{n}\int_{0}^{T}\|\nabla\mathbf{v}_i^{(m)}
(\cdot,t)\|_{2}^2\,dt
\leq
\\
&
C\left(\sum\limits_{i=1}^{n}\left(\|\mathbf{v}_{i,0}\|_{2}^2+\|\nabla\mathbf{v}_{i,0}\|_{2}^2\right)+
\sum\limits_{i=1}^{n}
\int_{0}^{T}\|\mathbf{f}_i(\cdot,t)\|_{2}^2\,dt\right):=\Xi_1.
\end{split}
\end{align}
\end{lemma}
\begin{proof}
For each $i\in\{1,\ldots,n\}$, { and at any time $t$,} we multiply
the $k$-th equation of \eqref{eq.2.15} by $2c_{i,k}^m(\cdot,t)$ and add up
the resulting equation, first from $k=1$ till $k=m$, and then sum
from $i=1$ till $i=n$, to obtain
\begin{equation}
\label{eq.2.29}%
\begin{split}
&
\frac{d}{dt}\left(\sum\limits_{i=1}^n
\int_{\Omega}\rho_i^{(m)}(\mathbf{x},t)|\mathbf{v}_i^{(m)}(\mathbf{x},t)|^2\,d\mathbf{x}
+\sum\limits_{i,j=1}^n\kappa_{ij}\int_{\Omega}\nabla
\mathbf{v}_i^{(m)}(\mathbf{x},t):\nabla\mathbf{v}_j^{(m)}(\mathbf{x},t)
\,d\mathbf{x} \right)
\\
&
+\sum\limits_{i=1}^n\int_{\Omega}\partial_t\rho_i^{(m)}(\mathbf{x},t)|\mathbf{v}_i^{(m)}(\mathbf{x},t)|^2\,d\mathbf{x}
+2\sum\limits_{i,j=1}^n\mu_{ij}\int_{\Omega}\nabla\mathbf{v}_i^{(m)}(\mathbf{x},t):\nabla\mathbf{v}_j^{(m)}(\mathbf{x},t)
\,d\mathbf{x}
\\
& +\sum\limits_{i=1}^{n}\int_\Omega
2\mathbf{div}\left(\rho_i^{(m)}(\mathbf{x},t)
\mathbf{v}_i^{(m)}(\mathbf{x},t)\otimes\mathbf{v}_i^{(m)}(\mathbf{x},t)\right)\cdot\mathbf{v}_i^{(m)}(\mathbf{x},t)\,d\mathbf{x}=
\\
& 2\sum\limits_{i=1}^{n}\int_\Omega
\rho_i^{(m)}(\mathbf{x},t)\mathbf{f}_i(\mathbf{x},t)\cdot\mathbf{v}_i^{(m)}(\mathbf{x},t)\,d\mathbf{x}-\sum\limits_{i,j=1}^{n}\gamma_{ij}\int_{\Omega}\left|\mathbf{v}_i^{(m)}(\mathbf{x},t)-\mathbf{v}_j^{(m)}(\mathbf{x},t)\right|^2\,d\mathbf{x}.
\end{split}
\end{equation}

Here, we have used the following formulas
\begin{alignat*}{2}
&\partial_t\left(\rho_i^{(m)}(\mathbf{x},t)|\mathbf{v}_i^{(m)}(\mathbf{x},t)|^2\right)=
2\rho_i^{(m)}(\mathbf{x},t) \mathbf{v}_i^{(m)}(\mathbf{x},t)\cdot
\partial_t\mathbf{v}_i^{(m)}(\mathbf{x},t)+\partial_t\rho_i^{(m)}(\mathbf{x},t)
|\mathbf{v}_i^{(m)}(\mathbf{x},t)|^2,
\\
&2\mathbf{v}_i^{(m)}(\mathbf{x},t)\cdot\partial_t
\left(\rho_i^{(m)}(\mathbf{x},t)
\mathbf{v}_i^{(m)}(\mathbf{x},t)\right)= 2\rho_i^{(m)}(\mathbf{x},t)
\mathbf{v}_i^{(m)}(\mathbf{x},t)\cdot
\partial_t\mathbf{v}_i^{(m)}(\cdot,t)
\\
&+2\partial_t\rho_i^{(m)}(\mathbf{x},t)
|\mathbf{v}_i^{(m)}(\mathbf{x},t)|^2 =
\partial_t\left(\rho_i^{(m)}(\mathbf{x},t)|\mathbf{v}_i^{(m)}(\mathbf{x},t)|^2\right)
+\partial_t\rho_i^{(m)}(\mathbf{x},t)|\mathbf{v}_i^{(m)}(\mathbf{x},t)|^2,
\\
& \int_{\Omega} \sum\limits_{i=1}^{n}\nabla\mathbf{v}_i^{(m)}(\mathbf{x},t):
\sum\limits_{j=1}^{n}\kappa_{ij}\nabla\partial_t\mathbf{v}_{j}^{(m)}(\mathbf{x},t)\,d\mathbf{x}
= \frac12\frac{d}{dt}\int_{\Omega}\sum\limits_{i,j=1}^{n}\kappa_{ij}
\nabla\mathbf{v}_i^{(m)}(\mathbf{x},t) :\nabla\mathbf{v}_j^{(m)}(\mathbf{x},t)\,d\mathbf{x},
\\
&
2\sum\limits_{i,j=1}^{n}\gamma_{ij}\int_{\Omega}\left(\mathbf{v}_j^{(m)}(\mathbf{x},t)-\mathbf{v}_i^{(m)}(\mathbf{x},t)\right)
\cdot\mathbf{v}_i^{(m)}(\mathbf{x},t)\,d\mathbf{x}=-\sum\limits_{i,j=1}^{n}\gamma_{ij}\int_{\Omega}\left|\mathbf{v}_i^{(m)}(\mathbf{x},t)-\mathbf{v}_j^{(m)}(\mathbf{x},t)\right|^2\,d\mathbf{x},
\end{alignat*}
where $\kappa_{ij}=\kappa_{ji}$, $\gamma_{ij}=\gamma_{ji}$.
Now, at any time $t$, we multiply the equation \eqref{eq.2.16} by
$|\mathbf{v}_i^{(m)}(\mathbf{x},t)|^2$ and integrate over $\Omega$ so that
\begin{equation}\label{eq.2.30}
\begin{split}
&\sum\limits_{i=1}^{n}\int_\Omega\partial_t\rho_i^{(m)}(\mathbf{x},t)|\mathbf{v}_i^{(m)}(\mathbf{x},t)|^2\,d\mathbf{x}+
\sum\limits_{i=1}^{n}\int_\Omega\mathrm{div}\left(\rho_i^{(m)}(\mathbf{x},t)\mathbf{v}_i^{(m)}(\mathbf{x},t)\right)|\mathbf{v}_i^{(m)}(\mathbf{x},t)|^2\,d\mathbf{x}=0.
\end{split}
\end{equation}

With the help of the following formulas
\begin{align}
\label{eq.A}
\begin{split}
&
\mathbf{div}\left(\rho_i^{(m)}\mathbf{v}_i^{(m)}\otimes\mathbf{v}_i^{(m)}\right)\cdot\mathbf{v}_i^{(m)}
=\mathrm{div}\left(\rho_i^{(m)}\mathbf{v}_i^{(m)}\right)|\mathbf{v}_i^{(m)}|^2 +
\rho_i^{(m)}(\mathbf{v}_i^{(m)}\cdot{\nabla})\mathbf{v}_i^{(m)}\cdot\mathbf{v}_i^{(m)},
\\
& \mathrm{div}\left(\rho_i^{(m)}|\mathbf{v}_i^{(m)}|^2\mathbf{v}_i^{(m)}\right)=
\mathrm{div}\left(\rho_i^{(m)}\mathbf{v}_i^{(m)}\right)|\mathbf{v}_i^{(m)}|^2
+2\rho_i^{(m)}(\mathbf{v}_i^{(m)}\cdot{\nabla})\mathbf{v}_i^{(m)}\cdot
\mathbf{v}_i^{(m)},
\end{split}
\end{align}
we get
\begin{align}
\label{eq.AA}
\begin{split}
&2\int_\Omega\mathbf{div}\left(\rho_i^{(m)}\mathbf{v}_i^{(m)}\otimes\mathbf{v}_i^{(m)}\right)\cdot\mathbf{v}_i^{(m)}\,d\mathbf{x}=
 \int_\Omega\mathrm{div}\left(\rho_i^{(m)}|\mathbf{v}_i^{(m)}|^2\mathbf{v}_i^{(m)}\right)\,d\mathbf{x}
\\
&\quad+\int_\Omega\mathrm{div}\left(\rho_i^{(m)}\mathbf{v}_i^{(m)}\right)|\mathbf{v}_i^{(m)}|^2\,d\mathbf{x}=
\int_\Omega\mathrm{div}\left(\rho_i^{(m)}\mathbf{v}_i^{(m)}\right)|\mathbf{v}_i^{(m)}|^2\,d\mathbf{x}.
\end{split}
\end{align}

Subtracting \eqref{eq.2.30} from \eqref{eq.2.29}, we get
\begin{align}\label{eq.2.31}
\begin{split}
& \frac{d}{dt}\int_\Omega\left(
\sum\limits_{i=1}^{n}\rho_i^{(m)}(\mathbf{x},t)
|\mathbf{v}_i^{(m)}(\mathbf{x},t)|^2
+\sum\limits_{i,j=1}^{n} \kappa_{ij}\nabla\mathbf{v}_i^{(m)}
(\mathbf{x},t):\nabla\mathbf{v}_j^{(m)}(\mathbf{x},t)\right)\,d\mathbf{x}\\
& +
2\sum\limits_{i,j=1}^{n}\mu_{ij}\int_\Omega\nabla\mathbf{v}_i^{(m)}(\mathbf{x},t):\nabla\mathbf{v}_j^{(m)}(\mathbf{x},t)\,d\mathbf{x}
=
\\
& 2\sum\limits_{i=1}^{n}\int_\Omega
\rho_i^{(m)}(\mathbf{x},t)\mathbf{f}_i(\mathbf{x},t)\cdot\mathbf{v}_i^{(m)}(\mathbf{x},t)\,d\mathbf{x}
-\sum\limits_{i,j=1}^{n}\gamma_{ij}\int_{\Omega}\left|\mathbf{v}_i^{(m)}(\mathbf{x},t)-\mathbf{v}_j^{(m)}(\mathbf{x},t)\right|^2\,d\mathbf{x}.
\end{split}
\end{align}
Let us introduce the energy function
\begin{equation*}
Y_1(t)=\int_\Omega\left(
\sum\limits_{i=1}^{n}\rho_i^{(m)}(\mathbf{x},t)|\mathbf{v}_i^{(m)}(\mathbf{x},t)|^2+\sum\limits_{i,j=1}^{n}
\kappa_{ij}\nabla\mathbf{v}_i^{(m)}(\mathbf{x},t):\nabla\mathbf{v}_j^{(m)}(\mathbf{x},t)\right)\,d\mathbf{x}.
\end{equation*}

Observe that by using assumptions \eqref{eq.1.15},
\eqref{eq.1.24}--\eqref{eq.1.25}, \eqref{eq.2.02}, estimate
\eqref{eq.2.23}, the H\"older, Sobolev and Cauchy inequalities, one
has
\begin{alignat}{2}
\label{eq.2.32} &
2\sum\limits_{i,j=1}^{n}\mu_{ij}\int_\Omega\nabla\mathbf{v}_i^{(m)}(\mathbf{x},t):\nabla\mathbf{v}_j^{(m)}(\mathbf{x},t)\,d\mathbf{x}
\geq 2\mu^-
\sum\limits_{i=1}^{n}\|\nabla\mathbf{v}_i^{(m)}(\cdot,t)\|_{2}^2,
\\
\label{eq.2.33}
\begin{split}
&Y_1(t)\geq \sum\limits_{i,j=1}^{n}
\kappa_{ij}\int_\Omega\nabla\mathbf{v}_i^{(m)}(\mathbf{x},t):
\nabla\mathbf{v}_j^{(m)}(\mathbf{x},t)\,d\mathbf{x} \geq \kappa^-\sum\limits_{i=1}^{n}
\|\nabla\mathbf{v}_i^{(m)}(\cdot,t)\|_{2}^2,
\end{split}
\\
\label{eq.2.34} & 2\sum\limits_{i=1}^{n}\int_\Omega
\rho_i^{(m)}(\mathbf{x},t)\mathbf{f}_i(\mathbf{x},t)\cdot\mathbf{v}_i^{(m)}(\mathbf{x},t)\,d\mathbf{x} \leq
\rho^+\sum\limits_{i=1}^{n}\|\mathbf{f}_i(\cdot,t)\|_{2}^2+Y_1(t).
\end{alignat}
Using \eqref{eq.2.32}--\eqref{eq.2.34} into \eqref{eq.2.31}, we
arrive at the following ordinary differential inequality
\begin{equation*}
\begin{split}
&\frac{dY_{1}(t)}{dt}
+2\mu^{-}\sum\limits_{i=1}^{n}\|\nabla\mathbf{v}_i^{(m)}(\cdot,t)\|_{2}^2
\leq
Y_1(t)+\rho^+\sum\limits_{i=1}^{n}\|\mathbf{f}_i(\cdot,t)\|_{2}^2.
\end{split}
\end{equation*}
From here, we get
  \begin{equation}\label{eq.2.35}
Y_{1}(t)\leq e^t \rho^+\int_0^te^{-s} \sum\limits_{i=1}^{n}
\|\mathbf{f}_i(\cdot,s)\|_{2}^2\,ds+e^{t}Y_{1}(0).
  \end{equation}
Integrating \eqref{eq.2.31} with respect to $t$ and using
\eqref{eq.2.35}, we arrive at the estimate \eqref{eq.2.28}.
Lemma~\ref{Lem.2.2} is thus proved.
\end{proof}

\subsection{\bf Uniform estimates for $\partial_t\mathbf{v}_i^{(m)}$ and
$\nabla\partial_t\mathbf{v}_i^{(m)}$}
\begin{lemma} \label{Lem.2.3}
{ For each $i\in\{1,\ldots,n\}$, let $\mathbf{v}_i^{(m)}$ be the Galerkin
approximation considered in \eqref{eq.2.14}. If the conditions
of Theorem~\ref{Theorem.2.1} are satisfied, then} there exists a
positive constant
$C=C(\rho^\pm,\kappa^-,\mu^+,\gamma^+,T,\Xi_1)$ such that {
\begin{equation}\label{eq.2.36}
\begin{split}
&\sum\limits_{i=1}^{n}\int_0^T\left(\|\sqrt{\rho_i^{(m)}(\cdot,t)}\partial_t\mathbf{v}_i^{(m)}(\cdot,t)\|^2_{2}
+\kappa^+\|\nabla\partial_t\mathbf{v}_i^{(m)}(\cdot,t)\|_{2}^2\right)dt
\leq
\\
&C\left(1+\sum\limits_{i=1}^{n}\int_{0}^{T}\|\mathbf{f}_i(\cdot,t)\|_{2}^2\,dt
+\sum\limits_{i=1}^{n}\|\mathbf{v}_{i,0}\|_{2}^2\right):=\Xi_2.
\end{split}
\end{equation}
}
\end{lemma}
\begin{proof}
We multiply the $k$-th equation of \eqref{eq.2.15} by $2\frac{dc_{i,k}^{m}(\cdot,t)}{dt}$, and add up the resulting equation for
$k\in\{1,\ldots,m\}$ and $i\in\{1,\ldots,n\}$, to obtain
\begin{equation}\label{eq.2.37}
\begin{split}
&2\sum\limits_{i=1}^{n}\int_\Omega\rho_i^{(m)}(\mathbf{x},t)|\partial_t\mathbf{v}_i^{(m)}(\mathbf{x},t)|^2\,d\mathbf{x}
+
2\sum\limits_{i,j=1}^{n}\kappa_{ij}\int_\Omega\nabla\partial_t
\mathbf{v}_i^{(m)}(\mathbf{x},t):\nabla\partial_t\mathbf{v}_{j}^{(m)}(\mathbf{x},t)\,
d\mathbf{x}
=
\\
&-2\sum\limits_{i=1}^{n}\int_\Omega\mathbf{div}\left(\rho_i^{(m)}(\mathbf{x},t)\mathbf{v}_i^{(m)}(\mathbf{x},t)\otimes\mathbf{v}_i^{(m)}(\mathbf{x},t)\right)\cdot\partial_t\mathbf{v}_i^{(m)}(\mathbf{x},t)\,d\mathbf{x}
\\
&-2\sum\limits_{i,j=1}^{n}\mu_{ij}\int_\Omega
\nabla\mathbf{v}_j^{(m)}(\mathbf{x},t):\nabla\partial_t\mathbf{v}_i^{(m)}(\mathbf{x},t)\,d\mathbf{x}+
2\sum\limits_{i=1}^{n}\int_\Omega
\rho_i^{(m)}(\mathbf{x},t)\mathbf{f}_i(\mathbf{x},t)\cdot\partial_t
\mathbf{v}_i^{(m)}(\mathbf{x},t)\,d\mathbf{x}
\\
&+2\sum\limits_{i,j=1}^{n}\gamma_{ij}\int_{\Omega}\left(\mathbf{v}_j^{(m)}(\mathbf{x},t)-\mathbf{v}_i^{(m)}(\mathbf{x},t)\right)\cdot\partial_t\mathbf{v}_i^{(m)}(\mathbf{x},t)\,d\mathbf{x}.
\end{split}
\end{equation}
Using assumptions \eqref{eq.1.24}--\eqref{eq.1.25}, estimate
\eqref{eq.2.23}, the H\"older, Sobolev and Cauchy inequalities, we
can show that for $d\in\{2,3\}$ we have
\begin{equation}\label{eq.2.38}
\begin{split}
&\left|-2\sum\limits_{i=1}^{n}\int_\Omega
\rho^{(m)}_i(\mathbf{x},t)\,\mathbf{div}\left(\mathbf{v}^{(m)}_i(\mathbf{x},t)\otimes\mathbf{v}^{(m)}_i(\mathbf{x},t)\right)\cdot\partial_t\mathbf{v}_i^{(m)}(\mathbf{x},t)\,d\mathbf{x}
\right|\leq
\\
&\frac{2}{\kappa^-}\left(\rho^+\right)^2\sum\limits_{i=1}^{n}\|\nabla\mathbf{v}_i^{(m)}(\cdot,t)\|_{2}^4
+
\frac{1}{2}\sum\limits_{i,j=1}^{n}\kappa_{ij}\int_\Omega
\nabla\partial_t\mathbf{v}_i^{(m)}(\mathbf{x},t):\nabla\partial_t\mathbf{v}_{j}^{(m)}(\mathbf{x},t)\,d\mathbf{x}
\end{split}
\end{equation}
and
\begin{equation}\label{eq.2.39}
\begin{split}
&\left|-2\sum\limits_{i,j=1}^{n}\mu_{ij}\int_\Omega
\nabla\mathbf{v}_j^{(m)}(\mathbf{x},t):\nabla\partial_t\mathbf{v}_i^{(m)}(\mathbf{x},t)\,d\mathbf{x}\right|\leq
\frac{(\mu^{+})^2}{\kappa^-}
\sum\limits_{i=1}^{n}\|\nabla\mathbf{v}_i^{(m)}(\cdot,t)\|_{2}^2
\\
&\quad+\sum\limits_{i,j=1}^{n}\kappa_{ij}
\int_\Omega\nabla\partial_t\mathbf{v}_i^{(m)}(\mathbf{x},t)
:\nabla\partial_t\mathbf{v}_{j}^{(m)}(\mathbf{x},t)\,d\mathbf{x}.
\end{split}
\end{equation}
By the
H\"older and Cauchy inequalities,
\begin{equation}\label{eq.2.40}
\begin{split}
&2\sum\limits_{i=1}^{n}\int_\Omega
\rho_i^{(m)}(\mathbf{x},t)\mathbf{f}_i(\mathbf{x},t)\cdot\partial_t\mathbf{v}_i^{(m)}(\mathbf{x},t)\,d\mathbf{x}\leq
\rho^+\sum\limits_{i=1}^{n}\|\mathbf{f}_i(\cdot,t)\|_{2}^2
+\sum\limits_{i=1}^{n}\int_\Omega\rho_i^{(m)}(\mathbf{x},t)|\partial_t\mathbf{v}_i^{(m)}(\mathbf{x},t)|^2\,d\mathbf{x}.
\end{split}
\end{equation}
Finally, using assumption  \eqref{eq.2.02}, one has
\begin{equation}\label{eq.2.41}
\begin{split}
&2\sum\limits_{i,j=1}^{n}\gamma_{ij}\int_{\Omega}\left(\mathbf{v}_j^{(m)}(\mathbf{x},t)-\mathbf{v}_i^{(m)}(\mathbf{x},t)\right)\cdot\partial_t\mathbf{v}_i^{(m)}(\mathbf{x},t)\,d\mathbf{x}\leq
\\
&\frac{\gamma^+}{\delta}\sum\limits_{i=1}^{n}\|\mathbf{v}_i^{(m)}(\cdot,t)\|_{2}^2+\frac{\gamma^+\delta}{\rho^-}\sum\limits_{i=1}^{n}\|
\sqrt{\rho_i^{(m)}(\cdot,t)}\partial_t\mathbf{v}_i^{(m)}(\cdot,t)\|_{2}^2.
\end{split}
\end{equation}
Using \eqref{eq.2.38}--\eqref{eq.2.41} in \eqref{eq.2.37}, and
integrating the resulting inequality between $0$ and $t\in(0,T)$, we
obtain
\begin{equation}\label{eq.2.42}
\begin{split}
& \left(1-\frac{\gamma^+\delta}{\rho^-}\right)
\sum\limits_{i=1}^{n}\int_{0}^{t}\int_\Omega\rho_i^{(m)}|\partial_t\mathbf{v}_i^{(m)}|^2\,d\mathbf{x}
d\tau
+\sum\limits_{i,j=1}^n\kappa_{ij}\int_{0}^{t}\int_\Omega\nabla\partial_t\mathbf{v}_i^{(m)}:
\nabla\partial_t\mathbf{v}_{j}^{(m)}\,d\mathbf{x} d\tau
\\
\leq
&
\frac{2}{\kappa^-}(\rho^+)^2\sum\limits_{i=1}^{n}\int_{0}^{t}\|\nabla\mathbf{v}_i^{(m)}(\cdot,\tau)\|_{2}^4\,d\tau
+ \frac{2(\mu^+)^2}{\kappa^-}
\sum\limits_{i=1}^{n}\int_{0}^{t}\|\nabla\mathbf{v}_i^{(m)}(\cdot,\tau)\|_{2}^2\,d\tau\
+
\\
&\sum\limits_{i=1}^{n}\rho^+\int_{0}^{t}\|\mathbf{f}_i(\cdot,\tau)\|_{2}^2\,d\tau+
\frac{\gamma^+}{\delta}\sum\limits_{i=1}^{n}\int_{0}^{t}\|\mathbf{v}_i^{(m)}(\cdot,\tau)\|_{2}^2\,d\tau.
\end{split}
\end{equation}
Taking the supreme of \eqref{eq.2.42} in $[0,T]$, using
 the Minkowski inequality, \eqref{eq.2.19}, \eqref{eq.2.23}
 and \eqref{eq.2.28}, we achieve \eqref{eq.2.36}.
\end{proof}

\subsection{\bf Compactness arguments}\label{sec2.4}

From Lemmas \ref{Lem.2.2}--\ref{Lem.2.3}, and for each
$i\in\{1,\ldots,n\}$, we have:
\begin{alignat}{2}
\label{eq.2.43} &
\mbox{$\mathbf{v}_i^{(m)}$ is uniformly bounded in $\L^2(0,T;\WW^{1,2}_0(\Omega))\cap \L^\infty(0,T;\WW^{1,2}_0(\Omega))$}; \\
\label{eq.2.44}& \mbox{$\partial_t\mathbf{v}_i^{(m)}$ is uniformly
bounded in ${\L^2}(0,T;\WW^{1,2}_0(\Omega))$}.
\end{alignat}
Then, due to the compact embedding {$\WW^{1,2}_0(\Omega)
\hookrightarrow\hookrightarrow \LL^{2}(\Omega)$, we can use
\eqref{eq.2.43}--\eqref{eq.2.44}and the Aubin-Lions compactness
lemma so that, eventually passing to a subsequence,
\begin{equation}\label{eq.2.45}
\mathbf{v}_i^{(m)} \xrightarrow[m\to+\infty]{}\mathbf{v}_i \quad \mbox{in}\ \
\L^2(0,T;\LL^{2}(\Omega)),\qquad i\in\{1,\ldots,n\}.
\end{equation}

On the other hand, we note that from \eqref{eq.2.23}, and for each
$i\in\{1,\dots,n\}$,
\begin{equation}\label{eq.2.46}
\rho_i^{(m)}\ \ \mbox{is uniformly bounded in}\
\L^\lambda(0,T;\L^\lambda(\Omega))\quad \forall\
\lambda\in\left[1,+\infty\right].
\end{equation}
By using \eqref{eq.2.46} and the Banach-Alaoglu theorem, there holds
for each $i\in\{1,\dots,n\}$, eventually passing to a subsequence,
\begin{equation}\label{eq.2.47}
\rho_i^{(m)}\xrightharpoonup[m\to+\infty]{}\rho_i\quad \mbox{in}\ \
\L^\lambda(0,T;\L^\lambda(\Omega))\quad \forall\
\lambda\in\left(1,+\infty\right),
\end{equation}
where $\rho_i$ satisfies \eqref{eq.1.28}, and moreover
\begin{equation}\label{eq.2.48}
0<\rho^-\leq\rho_i(\mathbf{x},t)\leq\rho^+\quad \forall\, (\mathbf{x},t)\in Q_T,
\end{equation}
for any $i\in\{1,\dots,n\}$.

Now, using the equation \eqref{eq.2.16}, together with the estimates \eqref{eq.2.23} and \eqref{eq.2.28}, it can be proved
that
\begin{equation}\label{eq.2.49}
\partial_t\rho_i^{(m)}\ \ \mbox{is uniformly bounded in}\ \L^2(0,T;\W^{-1,2^\ast}(\Omega)),
\end{equation}
where $2^\ast$ is the Sobolev conjugate of $2$: $2^\ast=\frac{2d}{d-2}$ if $d>2$, and $2^\ast$ can be any real in the interval $[1,\infty)$ if $d=2$, or in $[1,\infty]$ if $d<2$.
Moreover the following compact and continuous embeddings hold for
all $\lambda$ such that $2^\ast\leq \lambda<\infty$, ,
\begin{equation}\label{eq.2.50}
\L^\lambda(\Omega)\hookrightarrow\hookrightarrow
\W^{-1,\lambda}(\Omega) \hookrightarrow \W^{-1,2^\ast}(\Omega).
\end{equation}
Then, \eqref{eq.2.46}, \eqref{eq.2.49} and \eqref{eq.2.50} allow us
to use the Aubin-Lions compactness lemma so that, for some
subsequence,
\begin{equation}\label{eq.2.51}
\rho_i^{(m)}\xrightharpoonup[m\to+\infty]{}\rho_i\quad\mbox{in}\ \
\C([0,T];\W^{-1,\lambda}(\Omega)),
\end{equation}
for $2^\ast\leq \lambda <\infty$. On the other hand, testing the
continuity equations \eqref{eq.1.28} and \eqref{eq.2.16} with
$\rho_i\chi_{Q_t}$ and $\rho^{(m)}_i\chi_{Q_t}$, respectively, where
$\chi_{Q_t}$ denotes the characteristic function of the cylinder
$Q_t:=\Omega\times[0,t]$, with $t\in(0,T)$, integrating the
resulting equations over $Q_T$, using the divergence theorem,
together with \eqref{eq.1.29}--\eqref{eq.1.30}, \eqref{eq.2.17} and
\eqref{eq.2.11}, one has
\begin{equation}\label{eq.2.52}
\|\rho_i(\cdot,t)\|_{2}^{2}=\|\rho_{i,0}\|_{2}^{2}\quad\mbox{and}\quad
\|\rho_i^{(m)}(\cdot,t)\|_{2}^{2}=\|\rho_{i,0}^{(m)}\|_{2}^{2}
\quad\forall\ t\in[0,T].
\end{equation}
Thus, applying \eqref{eq.2.47} and \eqref{eq.2.52}, together with
\eqref{eq.2.51}, we get for all $t\in[0,T]$
\begin{equation}\label{eq.2.53}
\begin{split}
&\|\rho_i^{(m)}(\cdot,t)-\rho_i(\cdot,t)\|_{2}^{2}=\|\rho_i^{(m)}(\cdot,t)\|_{2}^{2}-\|\rho_i(\cdot,t)\|_{2}^{2}
\\
&
+2\int_\Omega(\rho_i(\mathbf{x},t)-\rho_i^{(m)}(\mathbf{x},t))\rho_i(\mathbf{x},t)d\mathbf{x}
\xrightarrow[m\to+\infty]{} 0.
\end{split}
\end{equation}
Then, as a consequence of \eqref{eq.2.23} and \eqref{eq.2.53}, we
have
\begin{equation}\label{eq.2.54}
\|\rho_i^{(m)}(\cdot,t)-\rho_i(\cdot,t)\|_{\lambda}^{\lambda}\xrightarrow[m\to+\infty]{}
0   \quad\forall\ \lambda\in[1,+\infty),\quad i\in\{1,\dots,n\}.
\end{equation}
Hence, \eqref{eq.2.51} and \eqref{eq.2.54} assure that
\begin{equation}\label{eq.2.55}
\rho_i^{(m)}\xrightarrow[m\to+\infty]{} \rho_i\quad\mbox{in}\ \
\C([0,T];\L^{\lambda}(\Omega)) \quad\forall\
\lambda\in[1,+\infty),\quad i\in\{1,\dots,n\}.
\end{equation}

}

\subsection{\bf The limit passage as $m\longrightarrow+\infty$}\label{sec2.5}

Due to \eqref{eq.2.43}--\eqref{eq.2.44}, we can use the
Banach-Alaoglu theorem so that, eventually passing to a subsequence,
\begin{alignat}{2}
\label{eq.2.56}&
\mathbf{v}_i^{(m)}\xrightharpoonup[m\to+\infty]{}\mathbf{v}_i \quad \mbox{in}\ \  \L^2(0,T;\WW^{1,2}_0(\Omega)), \\
\label{eq.2.57} &
\partial_t\mathbf{v}_i^{(m)}\xrightharpoonup[m\to+\infty]{}\partial_t\mathbf{v}_i \quad \mbox{in}\ \  \L^2(0,T;\WW^{1,2}_0(\Omega))
\end{alignat}
and
\begin{alignat}{2}
\label{eq.2.58} &
\mathbf{v}_i^{(m)}\xrightharpoonup[m\to+\infty]{\star}\mathbf{v}_i
\quad \mbox{in}\ \  \L^\infty(0,T;\WW^{1,2}_0(\Omega)).
\end{alignat}

As a consequence of \eqref{eq.2.45}, \eqref{eq.2.55} and
\eqref{eq.2.57}, we have
\begin{alignat}{2} \label{eq.2.59}
& \rho_i^{(m)}\partial_t\mathbf{v}_i^{(m)}\xrightharpoonup[m\to+\infty]{}\rho_i\partial_t\mathbf{v}_i\ \mbox{ in }\L^2(0,T;\LL^{2}(\Omega)), && \\
\label{eq.2.60}
& \sqrt{\rho_i^{(m)}}\partial_t\mathbf{v}_i^{(m)}\xrightharpoonup[m\to+\infty]{}\sqrt{\rho_i}\partial_t\mathbf{v}_i\ \mbox{ in }\L^2(0,T;\LL^{2}(\Omega)), && \\
\label{eq.2.61} &
\rho_i^{(m)}\mathbf{v}_i^{(m)}\xrightarrow[m\to+\infty]{}\rho_i\mathbf{v}_i\
\mbox{ in }\L^r(0,T;\LL^r(\Omega)),\qquad 1\leq r<2^\ast. &&
\end{alignat}

Gathering the information of \eqref{eq.2.28}, \eqref{eq.2.36} and
\eqref{eq.2.55}, we can prove (see~\cite[p.~14]{Antontsev-2021} for
details)  that
\begin{equation}\label{eq.2.62}
\rho_i^{(m)}(\mathbf{v}_i^{(m)}\cdot\nabla
)\mathbf{v}_i^{(m)}\xrightarrow[m\to+\infty]{}\rho_i(\mathbf{v}_i\cdot\nabla
) \mathbf{v}_i\ \text{ in }\ \L^{1}(0,T;\LL^1(\Omega)).
\end{equation}

For each $i\in \{1,\dots,n\}$, let now $\zeta_i$ be a continuously
differentiable function on $(0,T)$ such that $\zeta_i(\cdot,t)=0$.
Multiplying \eqref{eq.2.15}${}^\prime$) by $\zeta_i$ and integrating
the resulting equation between $0$ and $T$, we obtain
\begin{equation}\label{eq.2.63}
\begin{split}
&
\int_0^T\!\!\int_{\Omega}\rho_i^{(m)}\partial_t\mathbf{v}_i^{(m)}\cdot\boldsymbol{\psi}_k\zeta_i\,d\mathbf{x}
dt +
\int_0^T\!\!\int_{\Omega}\rho_i^{(m)}(\mathbf{v}_i^{(m)}\cdot\nabla
)\mathbf{v}_i^{(m)}\cdot\boldsymbol{\psi}_k\zeta_i\,d\mathbf{x} dt
\\
& -\sum\limits_{j=1}^n \kappa_{ij}\int_0^T\!\!\int_{\Omega}\nabla
\mathbf{v}_j^{(m)}:\nabla
\boldsymbol{\psi}_k\zeta_i^\prime\,d\mathbf{x} dt +\sum\limits_{j=1}^n
\mu_{ij}\int_0^T\!\!\int_{\Omega}\nabla
\mathbf{v}_j^{(m)}:\nabla \boldsymbol{\psi}_k\zeta_i\,d\mathbf{x} dt =
\\
& \sum\limits_{j=1}^n\kappa_{ij}\zeta_i(0)\int_\Omega\nabla
\mathbf{v}_j^{(m)}(\mathbf{x},0):\nabla \boldsymbol{\psi}_k\,d\mathbf{x} +
\int_0^T\!\!\int_{\Omega}\rho_i^{(m)}\mathbf{f}_i\cdot\boldsymbol{\psi}_k\zeta_i\,d\mathbf{x}
dt
\\
&+\sum\limits_{j=1}^{n}\gamma_{ij}\int_0^T\int_{\Omega}\left(\mathbf{v}_j^{(m)}(\mathbf{x},t)-\mathbf{v}_i^{(m)}(\mathbf{x},t)\right)\cdot\boldsymbol\psi_k\zeta_i\,d\mathbf{x}
dt, \qquad k\in\{1,\dots,m\}.
\end{split}
\end{equation}
For each $i\in \{1,\dots,n\}$, we consider $\zeta_i$ as above and
$\phi_i\in \C^\infty_0(\Omega)$. We then multiply \eqref{eq.2.16} by
$\eta=\phi_i\,\zeta_i$, and integrate the resulting equation over
$Q_T$ so that
\begin{equation}\label{eq.2.64}
- \int_0^T\!\!\int_{\Omega}\rho_i^{(m)}\phi_i\zeta_i'\,d\mathbf{x} dt -
\int_0^T\!\!\int_{\Omega}\rho_i^{(m)}\mathbf{v}_i^{(m)}\cdot\nabla
\phi_i\zeta_i\,d\mathbf{x} dt =
\zeta_i(0)\int_{\Omega}\rho_i^{(m)}(\mathbf{x},0)\phi_i\,d\mathbf{x}.
\end{equation}
Using \eqref{eq.2.17}$_1$, we pass \eqref{eq.2.63} to
the limit $m\to+\infty$ with the help of the hypothesis
\eqref{eq.2.18} and the convergence results \eqref{eq.2.55},
\eqref{eq.2.56}, \eqref{eq.2.59} and \eqref{eq.2.62}. To pass \eqref{eq.2.64} to the limit $m\to+\infty$, we use \eqref{eq.2.17}$_2$, \eqref{eq.2.55} and \eqref{eq.2.61}.

After all, we obtain for each $i\in \{1,\dots,n\}$
\begin{equation}\label{eq.2.65}
\begin{split}
&
\int_0^T\!\!\int_{\Omega}\rho_i\partial_t\mathbf{v}_i\cdot\boldsymbol{\psi}_k\zeta_i\,d\mathbf{x}
dt + \int_0^T\!\!\int_{\Omega}\rho_i(\mathbf{v}_i\cdot\nabla
)\mathbf{v}_i\cdot\boldsymbol{\psi}_k\zeta_i\,d\mathbf{x} dt
\\
& +\sum\limits_{j=1}^n\int_0^T\!\!\int_{\Omega}\nabla
\mathbf{v}_j:\nabla
\boldsymbol{\psi}_k(\mu_{ij}\zeta_i-\kappa_{ij}\zeta_i')\,d\mathbf{x} dt=
\sum\limits_{j=1}^n\zeta_i(0)\kappa_{ij}\int_\Omega\nabla
\mathbf{v}_{j,0}:\nabla \boldsymbol{\psi}_k \,d\mathbf{x}
\\
& +
\int_0^T\!\!\int_{\Omega}\rho_i\mathbf{f}_i\cdot\boldsymbol{\psi}_k\zeta_i\,d\mathbf{x}
dt
+\sum\limits_{j=1}^{n}\gamma_{ij}\int_0^T\int_{\Omega}\left(\mathbf{v}_j-\mathbf{v}_i\right)\cdot\boldsymbol\psi_k\zeta_i\,d\mathbf{x}
dt,
\end{split}
\end{equation}
where $k\in\{1,\dots,m\}$, and
\begin{equation}\label{eq.2.66}
-\int_0^T\!\!\int_{\Omega}\rho_i\phi_i\zeta_i'\,d\mathbf{x} dt
-\int_0^T\!\!\int_{\Omega}\rho_i\mathbf{v}_i\cdot\nabla\phi_i\zeta_i\,d\mathbf{x}
dt = \zeta_i(0)\int_{\Omega}\rho_{i,0}\phi_i\,d\mathbf{x}.
\end{equation}
By linearity, equation \eqref{eq.2.65} holds for any finite linear
combination of $\boldsymbol{\psi}_1,\dots,\boldsymbol{\psi}_m$, and,
by a continuity argument, it is still true for any
$\boldsymbol{\phi}_i\in\mathcal{V}$. Hence, writing \eqref{eq.2.65}
with $\zeta_i\in \C^\infty_0 (0,T)$, we can see that each
$\mathbf{v}_i$, for $i\in\{1,\dots,n\}$, satisfies
\begin{equation}\label{eq.2.67}
\begin{split}
& \int_{\Omega}\rho_i\left(\partial_t\mathbf{v}_i
+\big(\mathbf{v}_i\cdot\nabla
\big)\mathbf{v}_i\right)\cdot\boldsymbol{\phi}_i\,d\mathbf{x}
+\sum\limits_{j=1}^n \kappa_{ij}\int_{\Omega}\nabla
\partial_t\mathbf{v}_j:\nabla \boldsymbol{\phi}_i\,d\mathbf{x}
\\
&+\sum\limits_{j=1}^n\mu_{ij}\int_{\Omega}\nabla
\mathbf{v}_j:\nabla \boldsymbol{\phi}_i\,d\mathbf{x} =
\int_{\Omega}\rho_i\mathbf{f}_i\cdot\boldsymbol{\phi}_i\,d\mathbf{x}+\sum\limits_{j=1}^{n}\gamma_{ij}\int_{\Omega}\left(\mathbf{v}_j-\mathbf{v}_i\right)\cdot\boldsymbol{\phi}_i\,d\mathbf{x}\end{split}
\end{equation} %
in the sense of distributions on $(0,T)$. Note that \eqref{eq.2.67}
is the non-conservative form of \eqref{eq.2.04}. By the same
reasoning, we can infer, from \eqref{eq.2.66}, that each $\rho_i$,
for $i\in\{1,\dots,n\}$, satisfies \eqref{eq.2.05} in the sense of
distributions on $(0,T)$.

Finally, by the lower semicontinuity of the norms, we can combine the estimates \eqref{eq.2.28} and \eqref{eq.2.36} with the
convergence results \eqref{eq.2.57}, \eqref{eq.2.58} and
\eqref{eq.2.60} to show \eqref{eq.2.06} and \eqref{eq.2.07}.

\subsection{\bf Continuity in time}
\label{Sec.2.6}
 Due to \eqref{eq.2.55}, $\rho_i\in
C([0,T];\L^{\lambda}(\Omega))$ for all $\lambda\geq 1$. As a
consequence, the identity $\rho_i(0)=\rho_{i,0}$
(cf.~\eqref{eq.1.29}$_1$) is meaningful for any $i\in\{1,\dots,n\}$.
On the other hand, from \eqref{eq.2.06}--\eqref{eq.2.07} and for
each $i\in\{1,\dots,n\}$, $\mathbf{v}_i\in
\L^\infty(0,T;\WW^{1,2}_0(\Omega))$ and
$\partial_t\mathbf{v}_i\in \L^2(0,T;\WW^{1,2}_0(\Omega))$.
Hence,
\begin{equation}\label{eq.2.68}
\mathbf{v}_i\in C([0,T];\WW^{1,2}_0(\Omega)),\qquad i\in\{1,\dots,n\}. 
\end{equation}
Due to \eqref{eq.2.23} and Sobolev's inequality, there holds for
each $i\in\{1,\dots,n\}$
\begin{equation}\label{eq.2.69}
\|\rho_i(\cdot,t)\mathbf{v}_i(\cdot,t)\|_{2}\leq C\|\nabla
\mathbf{v}_i(\cdot,t)\|_{2}\quad \forall\ t\in[0,T].
\end{equation}
Combining \eqref{eq.2.68} with \eqref{eq.2.69}, we have
\begin{equation*}\label{r:v:C(0,T;V):2}
\rho_i\mathbf{v}_i\in C([0,T];\LL^{2}(\Omega)),\qquad
i\in\{1,\dots,n\},
\end{equation*}
and, as a consequence, we see the identity
$\big(\rho_i\mathbf{v}_i\big)(0)=\rho_{i,0}\mathbf{v}_{i,0}$
is meaningful. The proof of Theorem~\ref{Theorem.2.1} is thus
concluded.
\end{proof}

\section{Existence of unique pressures}\label{Sect:Rec-Press}
In this section, we recover the pressure $\pi_i$, for each $i\in\{1,\dots,n\}$, from the weak
formulation \eqref{eq.2.04} of the problem
\eqref{eq.1.26}--\eqref{eq.1.30}.
In what follows, $\C_{\rm{w}}([0,T];\L^{r'}(\Omega))$,
denotes the space of weakly continuous functions from $[0,T]$ into
$\L^{r'}(\Omega)$.

The following variant of de~Rham's lemma is of fundamental
importance to recover the pressures $\pi_i$ from the weak
formulation \eqref{eq.2.04}.

\begin{lemma}\label{lemm:BP}
Let $1<q<\infty$ and $\uu^{\ast}\in \WW^{-1,q'}(\Omega)$. If
\begin{equation*}
\left\langle \uu^{\ast},\uu\right\rangle=0 \quad\forall\ \uu\in
\mathbf{V},
\end{equation*}
then there exists a unique $\pi\in \L^{q'}(\Omega)$, with
$\int_{\Omega}\pi\,d\mathbf{x}=0$, such that
\begin{equation*}
\left\langle \uu^{\ast},\uu\right\rangle = \int_{\Omega}\pi
\operatorname{div}\uu\,d\mathbf{x}\qquad\forall\,\uu\in
\WW^{1,q}_0(\Omega).
\end{equation*}
Moreover, there exists a positive constant $C$ such that
\begin{equation*}
\|\pi\|_{q'}\leq C \|\uu^{\ast}\|_{-1,q'}.
\end{equation*}
\end{lemma}

\begin{proof}
The proof combines the results of \cite{Bogovskii:1980} and
 \cite{Pileckas:1980}, see also Theorems~III.3.1 and III.5.3
of the monograph~\cite{Galdi:2011}.
\end{proof}

The following auxiliary result is directly connected with the
previous lemma, and is also important to prove the existence of a
unique pressure.
\begin{lemma}\label{Lemma-Beg}
Let be given $\xi\in \L^{r}(\Omega)$, with $1<r<\infty$, and such
that $\int_{\Omega}\xi\,d\mathbf{x}=0$. Then there exists a solution
$\ww\in \WW^{1,r}_0(\Omega)$ to the problem
\begin{equation*}
\mathrm{div}\ww=\xi\quad\mbox{in}\quad \Omega,
\end{equation*}
and such that
\begin{equation*}
\|\nabla\ww\|_{r}\leq C\|\xi\|_{r}
\end{equation*}
for some positive constant $C$.
\end{lemma}
\begin{proof}
The proof is due to \cite{Bogovskii:1980} (see
also~\cite[Theorem~III.3.1]{Galdi:2011}).
\end{proof}

Lemmas~\ref{lemm:BP} and~\ref{Lemma-Beg} are used to prove the main
result of this section.

\begin{theorem}\label{thm:p}
For each $i\in\{1,\dots,n\}$, let $(\rho_i,\mathbf{v}_i)$ be a weak
solution to the problem \eqref{eq.1.26}--\eqref{eq.1.30} in the
conditions of Theorem~\ref{Theorem.2.1}. Then, for  each
$i\in\{1,\dots,n\}$, there exists a unique function $\pi_i\in
\C_{\rm{w}}([0,T];\L^{r^{\prime}}(\Omega))$, with
$\int_{\Omega}\pi_i(\cdot,t)\,d\mathbf{x}=0$ for all $t\in[0,T]$, and $r$
satisfying \eqref{eq.3.05} below, such that
\begin{equation}
\label{eq.3.01}%
\begin{split}
\frac{d}{dt}&\int_{\Omega}\left(\rho_i(\mathbf{x},t)\mathbf{v}_i(\mathbf{x},t)\cdot\boldsymbol{\phi}_i
+
\sum\limits_{j=1}^{n}\kappa_{ij}\nabla\mathbf{v}_j(\mathbf{x},t):\nabla\boldsymbol{\phi}_i\right)d\mathbf{x}
\\
&
+\int_\Omega\left(\sum\limits_{j=1}^{n}\mu_{ij}\nabla\mathbf{v}_j(\mathbf{x},t)-\rho_i(\mathbf{x},t)\mathbf{v}_i(\mathbf{x},t)\otimes\mathbf{v}_i(\mathbf{x},t)\right):\nabla\boldsymbol{\phi}_i\,d\mathbf{x}
\\
&
-\sum\limits_{j=1}^{n}\gamma_{ij}\int_{\Omega}\left(\mathbf{v}_j(\mathbf{x},t)-\mathbf{v}_i(\mathbf{x},t)\right)\cdot\boldsymbol{\phi}_i\,d\mathbf{x}
-\int_\Omega\rho_i(\mathbf{x},t)\mathbf{f} _i(\mathbf{x},t)\cdot\boldsymbol{\phi}_i\,d\mathbf{x}
\\
&
=\int_{\Omega}\pi_i(\mathbf{x},t)\operatorname{div}\!\boldsymbol{\phi}_i\,d\mathbf{x}
\end{split}
\end{equation}
for all $\boldsymbol{\phi}_i\in \WW_{0}^{1,r}(\Omega)$, in the distribution sense on $(0,T)$. Moreover,
there exists a positive constant $C$ such that
\begin{equation}
\label{eq.3.02}%
\begin{split}
& \sup_{t\in[0,T]} \|\pi_i(\cdot,t)\|_{r'}\leq  C.
\end{split}
\end{equation}
\end{theorem}

\begin{proof}
For each $i\in\{1,\dots,n\}$, we multiply \eqref{eq.2.04} by
$\xi_i\in \C^\infty_0(0,T)$ and integrate the resulting equation over
$(0,T)$. Then we get
\begin{equation*}
-\int_0^Ta_i\,\xi_i'\,dt=\int_0^Tb_i\,\xi_i\,dt,
\end{equation*}
where $\xi_i'$ denotes the time derivative of $\xi_i$,
\begin{alignat}{2}
\label{eq.3.03} &
a_i(t):=\int_{\Omega}\left(\rho_i(\mathbf{x},t)\mathbf{v}_i(\mathbf{x},t)\cdot\boldsymbol{\psi}_i + \sum\limits_{j=1}^{n}\kappa_{ij}\nabla\mathbf{v}_j(\mathbf{x},t):\nabla\boldsymbol{\psi}_i\right)d\mathbf{x}, &&  \\
\label{eq.3.04} &
b_i(t):=-\int_\Omega\mathbf{Q}_i(\mathbf{x},t):\nabla\boldsymbol{\psi}_i\,d\mathbf{x},
\end{alignat}
and
\begin{equation*}
\mathbf{Q}_i(\mathbf{x},t):=
\sum\limits_{j=1}^{n}\mu_{ij}\nabla\mathbf{v}_j(\mathbf{x},t)-\rho_i(\mathbf{x},t)\mathbf{v}_i(\mathbf{x},t)\otimes\mathbf{v}_i(\mathbf{x},t)
+\sum\limits_{j=1}^{n}\gamma_{ij}\mathbf{G}_{ij}(\mathbf{x},t)+\mathbf{F}_i(\mathbf{x},t),\quad
t\in [0,T],
\end{equation*}
with
\begin{alignat*}{2}
&
\mathbf{div}\,\mathbf{G}_{ij}(\mathbf{x},t)=\left(\mathbf{v}_j(\mathbf{x},t)-\mathbf{v}_i(\mathbf{x},t)\right),\quad t\in [0,T],
\\
&
\mathbf{div}\mathbf{F}_i(\mathbf{x},t)=\rho_i(\mathbf{x},t)\mathbf{f}_i(\mathbf{x},t),\quad t\in [0,T],
\end{alignat*}

We claim that
\begin{equation}\label{eq.3.05}
\mathbf{Q}_i\in \L^{r'}(0,T;\LL^{r'}(\Omega)),\quad
r\geq\max\left\{2,\ \frac{d}{2} \right\}.
\end{equation}
Note that $\LL^{r'}(\Omega)$ denotes here the Lebesgue space of tensor-valued functions.

Using \eqref{eq.1.24}--\eqref{eq.1.25},
\eqref{eq.2.03}, \eqref{eq.2.06}, \eqref{eq.2.48} and the
H\"{o}lder, Minkowski and Sobolev inequalities, we have for each
$i\in\{1,\dots,n\}$
\begin{alignat}{4}
& \label{eq.3.06}
\int_0^T\left\|\sum\limits_{j=1}^{n}\mu_{ij}\nabla\mathbf{v}_j(\cdot,t)\right\|_{r'}^{r'}dt\leq
\sum\limits_{j=1}^{n}\mu_{ij}\int_0^T\|\nabla\mathbf{v}_j(\cdot,t)\|_{r'}^{r'}dt \leq C_1,\quad r\geq 2, && \\
&
\int_0^T\|\rho_i(\cdot,t)\mathbf{v}_i(\cdot,t)\otimes\mathbf{v}_i(\cdot,t)\|
_{r'}^{r'}\,dt\leq C_2'\int_0^T\|\mathbf{v}_i(\cdot,t)\|^{2r'}_{2r'}\,dt\leq C_2,\quad r\geq \frac{d}{2}, && \\
&
\int_0^T\left\|\sum\limits_{j=1}^{n}\gamma_{ij}\left(\mathbf{v}_j(\cdot,t)-\mathbf{v}_i(\cdot,t)\right)dt\right\|_{r'}^{r'}\,dt\leq
C_3'\sum\limits_{j=1}^{n}\gamma_{ij}\int_0^T\|\mathbf{v}_j(\cdot,t)\|_{r'}^{r'}\,dt
\leq C_3 \quad r\geq \frac{2d}{2+2}, &&
\\
& \label{eq.3.09}
\int_0^T\|\rho_i(\cdot,t)\mathbf{f}_i(\cdot,t)\|_{r'}^{r'}\,dt\leq
C_4'\int_0^T\|\mathbf{f}_i(\cdot,t)\|_{r'}^{r'}\,dt\leq C_4,\quad r\geq2,
&&
\end{alignat}
for some positive constants $C_1,\ldots,C_4$ . Therefore,
\eqref{eq.3.05} follows from \eqref{eq.3.06}--\eqref{eq.3.09}. In
turn, \eqref{eq.3.04} and \eqref{eq.3.05} imply that $b_i\in
\L^{r'}(0,T)$, and thus $a_i\in \W^{1,r'}(0,T)$, with $a_i'=b_i$, for
each $i\in\{1,\dots,n\}$. By Sobolev's embedding theorem, $a_i$ can
be represented by a continuous function which we still denote by
$a_i$. For each $i\in\{1,\dots,n\}$, we integrate the identity
$a'_i=b_i$ as follows:
\begin{equation}\label{eq.3.10}
a_i(t)=a_i(0)+\int_0^t b_i(s)\,ds\quad \forall\ t\in (0,T),\quad
i\in\{1,\dots,n\}.
\end{equation}
For each $i\in\{1,\dots,n\}$, we define the function
\begin{equation*}
\mathbf{R}_i(\mathbf{x},t):=\int_0^t\mathbf{Q}_i(\mathbf{x},s)\,ds,\quad
t\in[0,T].
\end{equation*}
For $t\in[0,T]$ arbitrarily chosen, we can use this function,
together with \eqref{eq.1.29} and \eqref{eq.3.03}--\eqref{eq.3.04},
to rewrite \eqref{eq.3.10} in the following way
\begin{equation*}
\begin{split}
\int_{\Omega}& \bigg[\rho_i(\mathbf{x},t)\mathbf{v}_i(\mathbf{x},t)\cdot\boldsymbol{\phi}_i
+\sum\limits_{j=1}^{n}\kappa_{ij}\nabla\mathbf{v}_j(\mathbf{x},t):
\nabla\boldsymbol{\phi}_i-\rho_{i,0}\mathbf{v}_{i,0}\cdot\boldsymbol{\phi}_i-\sum\limits_{j=1}^{n}\kappa_{ij}\nabla\mathbf{v}_{j,0}:\nabla\boldsymbol{\phi}_i \\
&\ \ + \mathbf{R}_i(\mathbf{x},t):\nabla\boldsymbol{\phi}_i\bigg]\,d\mathbf{x}=0.
\end{split}
\end{equation*}
For each $i\in\{1,\dots,n\}$, we use Lemma~\ref{lemm:BP}, with
$q=r'$, where $r$ is given by \eqref{eq.3.05}, and
\begin{equation*}
\uu^\ast_i(\mathbf{x},t):=\rho_i(\mathbf{x},t)\mathbf{v}_i(\mathbf{x},t)-\rho_{i,0}\mathbf{v}_{i,0}-\sum\limits_{j=1}^{n}\kappa_{ij}\Delta\left(\mathbf{v}_j(\mathbf{x},t)-\mathbf{v}_{j,0}\right)-
\mathbf{div}(\mathbf{R}_i(\mathbf{x},t)),
\end{equation*}
to ensure the existence of a unique function $\pi_i(\mathbf{x},t)\in
\L^{r'}(\Omega)$, with $\int_\Omega\pi_i(\mathbf{x},t)\,d\mathbf{x}=0$, such that
\begin{equation}\label{eq.3.11}
\begin{split}
\int_{\Omega}& \bigg[\rho_i(\mathbf{x},t)\mathbf{v}_i(\mathbf{x},t)\cdot\boldsymbol{\phi}_i
+\sum\limits_{j=1}^{n}\kappa_{ij}\nabla\mathbf{v}_j(\mathbf{x},t):\nabla\boldsymbol{\phi}_i-
\left(\rho_{i,0}\mathbf{v}_{i,0}\cdot\boldsymbol{\phi}_i
+\sum\limits_{j=1}^{n}\kappa_{ij}\nabla\mathbf{v}_{j,0}:\nabla\boldsymbol{\phi}_i\right)
\\
&\,\,+\mathbf{R}_i(\mathbf{x},t):\nabla\boldsymbol{\phi}_i\bigg]\,d\mathbf{x}=
\int_{\Omega}\pi_i(\mathbf{x},t)\mathrm{div}\boldsymbol{\phi}_i\,d\mathbf{x}\qquad
\forall\,\boldsymbol{\phi}_i\in W_0^{1,r'}(\Omega)^d.
\end{split}
\end{equation}
In addition, for each $i\in\{1,\dots,n\}$, there exists a positive
constant $C_i$ such that
\begin{equation*}
\begin{split}
& \|\pi_i(\cdot,t)\|_{r'}\leq \\
& C_i\left(\|\rho_i(\cdot,t)\mathbf{v}_i(\cdot,t)-\rho_{i,0}(\cdot)\mathbf{v}_{i,0}(\cdot)\|_{r'} +
\sum\limits_{j=1}^{n}\kappa_{ij}\|\nabla(\mathbf{v}_j(\cdot,t)-\mathbf{v}_{j,0}(\cdot))\|_{r'}
+ \|\mathbf{R}_i(\cdot,t)\|_{r'}\right).
\end{split}
\end{equation*}

We claim now that, for each $i\in\{1,\dots,n\}$,
$$t\longmapsto \pi_i(\cdot,t)$$
is a Bochner-measurable function. From Lemma~\ref{Lemma-Beg}, we
know that,  given $\xi_i\in \L^{r}(\Omega)$, there exists
$\boldsymbol{\psi}_i\in W_0^{1,r}(\Omega)^d$ such that
\[
\mathrm{div}\boldsymbol{\phi}_i=\xi_i-\xi_{i,\Omega},\qquad
\xi_{i,\Omega}:=\frac{1}{|\Omega|}\int_{\Omega}\xi_i\,d\mathbf{x},
\]
where $|\Omega|$ denotes the $d$-Lebesgue measure of $\Omega$.
Combining this result with \eqref{eq.3.11}, and still using the fact
that $\int_\Omega \pi_i(\cdot,t)\,d\mathbf{x}=0$, we see that, for each
$i\in\{1,\dots,n\}$, and for an arbitrarily chosen $\xi_i\in
\L^{r}(\Omega)$,
\begin{equation}\label{eq.3.12}
\begin{split}
\int_{\Omega} \pi_i(\mathbf{x},t)\,\xi_i\,d\mathbf{x} & =\int_{\Omega}
\pi_i(\mathbf{x},t)\,(\xi_i-\xi_{i,\Omega})\,d\mathbf{x}=
\int_{\Omega}\Big(\Big(\rho_i(\mathbf{x},t)\mathbf{v}_i(\mathbf{x},t)-\rho_{i,0}\mathbf{v}_{i,0}\Big)\cdot\boldsymbol{\phi}_i
\\
&
+\sum\limits_{j=1}^{n}\kappa_{ij}\Big(\nabla\mathbf{v}_j(\mathbf{x},t)-\nabla\mathbf{v}_{j,0}\Big):\nabla\boldsymbol{\phi}_i
+ \mathbf{R}_i(\mathbf{x},t):\nabla\boldsymbol{\phi}_i\Big)\,d\mathbf{x}.
\end{split}
\end{equation}
From \eqref{r:v:C(0,T;V):2} and \eqref{eq.2.68}, we can infer that
the right-hand side function of \eqref{eq.3.12} is continuous in
time. Whence the function defined by the left-hand side of
\eqref{eq.3.12} is also continuous in time, which shows that
$\pi_i\in \C_{\rm{w}}([0,T];\L^{r'}(\Omega))$. Consequently, $\pi_i$ is
Bochner-measurable  for any $i\in\{1,\dots,n\}$.

Then, by successively replacing, in \eqref{eq.3.11},
$\mathbf{R}_i(\cdot,t)$, $\mathbf{Q}_i(\cdot,s)$ and
$\mathbf{F}_i(\cdot,s)$ by the corresponding expressions, we easily
get \eqref{eq.3.01}.

Next, for each $i\in\{1,\dots,n\}$, we use the Minkowski and
H\"{o}lder inequalities, together with Fubini's theorem, so that
from \eqref{eq.3.12} we obtain the following inequality
\begin{equation}\label{eq.3.13}
\begin{split}
\|\pi_i(\cdot,t)\|_{r'}\leq C_i & \left(
\|\rho_i(\cdot,t)\mathbf{v}_i(\cdot,t)-\rho_{i,0}\mathbf{v}_{i,0}\|_{r'} + \sum\limits_{j=1}^{n}\kappa_{ij}\|\nabla\mathbf{v}_j(\cdot,t)-\nabla\mathbf{v}_{j,0}\|_{r'} \right. \\
& \ \ \left.+ \left(\int_0^t
\|\mathbf{Q}_i(\cdot,s)\|_{r'}^{r'}\,ds\right)^{\frac{1}{r'}}\right),
\end{split}
\end{equation}
for some positive constant $C_i$. Taking the supreme of
\eqref{eq.3.13} in $[0,T]$, and using the estimates
\eqref{eq.3.06}--\eqref{eq.3.09}, we finally achieve
\eqref{eq.3.02}.
\end{proof}

\section{Regularity}\label{Sect:Regularity}

In this section, we prove that the solutions found in
Section~\ref{Sect:EWS} have higher regularity. This result will be
important to prove the uniqueness result. Before we do so, let us
recall the following classical result on the regularity of the
density $\rho$ for the nonhomogeneous Navier-Stokes equations. Since
the continuity equation is the same for either the nonhomogeneous
Navier-Stokes or Kelvin-Voigt equations, the following result is
useful here as well.

\begin{lemma}\label{lemma:LS}
For each $i\in\{1,\dots,n\}$, let $(\mathbf{v}_i,\rho_i)$ be a solution to
\eqref{eq.1.26}--\eqref{eq.1.30}  in the conditions of
Theorem~\ref{Theorem.2.1}. If, in addition, $\partial\Omega$ is of
class $\C^2$, and, for each $i\in\{1,\dots,n\}$, $\mathbf{v}_i\in
\C([0,T];\WW^{1,\infty}(\Omega))$ and
\begin{equation}\label{eq.4.01}
\rho_{i,0}\in \W^{1,\infty}(\Omega),
\end{equation}
then
\begin{alignat}{2}
& \label{eq.4.02} \|\nabla \rho_i(\cdot,t)\|_{\infty}\leq
\sqrt{d}\|\nabla \rho_{i,0}\|_{\infty}
\exp\left(\int_{0}^{t}\|\nabla
\mathbf{v}_i(\cdot,s)\|_{\infty}\,ds\right), &&
\\
& \label{eq.4.03} \|\partial_t\rho_i(\cdot,t)\|_{\infty}\leq
\sqrt{d}\|\nabla \rho_{i,0}\|_{\infty}
\|\mathbf{v}_i(\cdot,t)\|_{\infty} \exp\left(\int_{0}^{t}\|\nabla
\mathbf{v}_i(\cdot,s)\|_{\infty}\,ds\right) &&
\end{alignat}
for all $t\in[0,T]$.
\end{lemma}
{
\begin{proof}
For the proof, we refer the reader to \cite[Lemma~1.3]{LS:1978}.
\end{proof}

Next, according to Lemma~\ref{lemma:LS},  establishing higher  regularity of the density, in particularly,
that $\|\nabla\rho_i(\cdot,s)\|_{\infty}$ is bounded, we first need to prove the boundedness of
$\|\nabla\mathbf{v}_i(\cdot,s)\|_{\infty}$ for all $s\in[0,T]$. Before doing so, we first prove the following auxiliary result.
\begin{lemma}\label{lem.4.10}
For each $i\in\{1,\dots,n\}$, let $(\mathbf{v}_i,\rho_i)$ be a solution to the problem
\eqref{eq.1.26}--\eqref{eq.1.30},  in the conditions of
Theorem~\ref{Theorem.2.1} and Lemma ~\ref{lem:LS}.
Assume that there exists a real number $r$ such that $d<r<2^*$ and for each $i\in\{1,\dots,n\}$ there hold
\begin{alignat}{2}
& \label{eq.4.04} \mathbf{v}_{i,0}\in \WW^{2,r}(\Omega) \cap \HH_{0}^{2}(\Omega),
 && \\
& \label{eq.4.05} \mathbf{f}_i\in \L^{2}\left(0,T;\LL^r(\Omega)\right),\quad
d<r<2^*. &&
\end{alignat}
 Then   the estimate
\begin{equation*}
\sup_{t\in[0,T]}
\sum_{i=1}^{n}\left\Vert\Delta\mathbf{v}_i(\cdot,t)
\right\Vert_{r}^{r}\leq
K_3<\infty,\quad d<r<2^{*}
\end{equation*}
 is valid, for some positive constant $K_3$.
\end{lemma}
\begin{proof}
 Let us introduce the functions defined by
\begin{equation}\label{eq.4.09}
\sum_{j=1}^{n}\mu_{ij}\uu_{j}:=\sum_{j=1}^{n}\left(\mu_{ij}\mathbf{v}_j
+\kappa_{ij}\partial_t\mathbf{v}_j\right), \ i=1,2,\dots,n.
\end{equation}
In
the nonconservative form of the Kelvin-Voigt system formed by
\eqref{eq.1.26}--\eqref{eq.1.28}, we obtain for each
$i\in\{1,\dots,n\}$ and all $t\in(0,T)$
\begin{align*}
\begin{cases}
&\mathrm{div}\uu_i(\mathbf{x},t)=0\quad\mbox{ in }\ \Omega,
\\
&-\Delta\left(\sum\limits_{j=1}^n\mu_{ij}\uu_{j}(\mathbf{x},t)\right)+\nabla\pi_i(\mathbf{x},t)=
\mathbf{F}_i(\mathbf{x},t)\quad \mbox{
in }\ \Omega,
\\
& \uu_i(\mathbf{x},t)=\boldsymbol{0}\quad\mbox{ on }\ \partial\Omega,
\end{cases}
\end{align*}
where%
\begin{equation*}
\mathbf{F}_i:=\rho_i\Big(\mathbf{f}_i-\partial_t\mathbf{v}_i-(\mathbf{v}_i\cdot
\nabla) \mathbf{v}_i\Big)
+\sum\limits_{j=1}^{n}\gamma_{ij}\left(\mathbf{v}_j-\mathbf{v}_i\right).
\end{equation*}
By the theory of the stationary Stokes problem (see
\emph{e.g.}~\cite[Theorem IV.6.1]{Galdi:2011}), we get
\begin{equation} \label{eq.4.10}
\begin{split}
& \left\Vert
\Delta\left(\sum\limits_{j=1}^n{{\mu_{ij}}}\uu_{j}(\cdot,t)\right)\right\Vert_{r}
+ \|\nabla \pi_i(\cdot,t)\|_{r} \leq
C_1\left\Vert\mathbf{F}_i\right\Vert_{r}\\
&
=C_1\left\Vert\rho_i\big(\mathbf{f}_i-\partial_t\mathbf{v}_i-(\mathbf{v}_i\cdot
\nabla) \mathbf{v}_i\big)
+\sum\limits_{j=1}^{n}\gamma_{ij}\left(\mathbf{v}_j-\mathbf{v}_i\right)\right
\Vert_{r}
 \\
& \leq
C_2\left(\left\Vert\mathbf{f}_i(\cdot,t)\right\Vert_{r}+\left\Vert\partial_t\mathbf{v}_i(\cdot,t)\right\Vert_{r}+
\left\Vert\mathbf{v}_i(\cdot,t)\right\Vert_{q} \left\Vert\nabla
\mathbf{v}_i(\cdot,t)\right\Vert_{\frac{rq}{q-r}}+\gamma^+\sum\limits_{j=1}^{n}\left\Vert\mathbf{v}_j(\cdot,t)\right\Vert_{r} \right),
\end{split}
\end{equation}
for some $q>r$ and for all $t\in(0,T)$.
  Let us consider the linear algebraic system
\begin{equation*}
  \sum_{j=1}^{n}\mu_{ij}\Delta\uu_{j}=\mathbf{G}_i,\quad \left\Vert
 \mathbf{G}_i
\right\Vert_{r}\leq \left\Vert \mathbf{F}_i
\right\Vert_{r}.
\end{equation*}
Here we consider $\mathbf{G}_i$ as
given and $\Delta\uu_{j}$ as unknown. According to \eqref{eq.1.24}, the
matrix $\mathbf{M}$ has the inverse matrix $\mathbf{M}^{-1}$. This means that
\eqref{eq.4.10} implies
\begin{equation*}
\Delta
\uu_{j}=\mathbf{M}^{-1}\mathbf{G}_i,\qquad\left\Vert \mathbf{M}^{-1}\mathbf{G}_i
\right\Vert_{r}\leq C\left\Vert \mathbf{F}_i
\right\Vert_{r}\end{equation*}
 Thus the estimates \eqref{eq.4.10} imply
 \begin{equation}
 \label{eq.4.111}
\begin{split}
& \sum\limits_{j=1}^n \left\Vert \Delta\uu_{j}(\cdot,t)
\right\Vert_{r} \leq
C_3\sum\limits_{i=1}^n\left\Vert\rho_i\big(\mathbf{f}_i-\partial_t\mathbf{v}_i-(\mathbf{v}_i\cdot
\nabla) \mathbf{v}_i\big)
+\sum\limits_{j=1}^{n}\gamma_{ij}\left(\mathbf{v}_j-\mathbf{v}_i\right)\right
\Vert_{r}
 \\
& \leq
C_3\sum\limits_{i=1}^n\left(\left\Vert\mathbf{f}_i(\cdot,t)\right\Vert_{r}+\left\Vert\partial_t\mathbf{v}_i(\cdot,t)\right\Vert_{r}+
\left\Vert\mathbf{v}_i(\cdot,t)\right\Vert_{q} \left\Vert\nabla
\mathbf{v}_i(\cdot,t)\right\Vert_{\frac{rq}{q-r}}+\gamma^+\sum\limits_{j=1}^{n}\left\Vert\mathbf{v}_j
\right\Vert_{r} \right).
\end{split}
\end{equation}
 Moreover, using (see \cite[Lemma 1]{Antontsev-2021}
and \cite{Mazya:2011,CM:2016})
\begin{alignat}{2}\label{emb.4.1}
& \|\nabla\uu\|_{r^\ast} \leq C_{*}\|\Delta
\uu\|_{r}\qquad \forall\ \uu\in \WW^{2,r}(\Omega)\cap
\WW^{1,r}_0(\Omega), &&
\end{alignat}
we obtain
\begin{equation}\label{emb.4.2}
\left\Vert\nabla \mathbf{v}_i(\cdot,t)\right\Vert_{\frac{rq}{q-r}}\leq
C_0\left\Vert \Delta\mathbf{v}_i(\cdot,t)\right\Vert_{r},\qquad d<q<r\leq
q\leq 2^*,
\end{equation}
and by the Sobolev inequality, we have
\begin{equation}\label{emb.4.3}
\left\Vert\mathbf{v}_i(\cdot,t)\right\Vert_{q}\leq C_{0,1}\Vert\nabla
\mathbf{v}_i(\cdot,t)\Vert_{2}\leq C_{0,1} K_1<\infty,\qquad d<r<q
\leq
2^*,
\end{equation}
\begin{equation}\label{emb.4.4:1}
\left\Vert\mathbf{v}_i(\cdot,t)\right\Vert_{r}\leq C_{0,2}\Vert\nabla
\mathbf{v}_i(\cdot,t)\Vert_{2}\leq C_{0,2} K_1<\infty,\qquad r\leq2^*
\end{equation}
and
\begin{equation}\label{emb.4.4}
\left\Vert\partial_t\mathbf{v}_i(\cdot,t)\right\Vert_{r}\leq
C_{0,3}\Vert\nabla \partial_t\mathbf{v}_i(\cdot,t)\Vert_{2},\qquad
r\leq2^*
\end{equation}
for some positive constants $C_{0,i}$ depending on $d$, $\Omega$,
and $r$, where $K_1$ is the constant from \eqref{eq.2.06}. Combining
inequalities \eqref{emb.4.1}--\eqref{emb.4.4} with \eqref{eq.4.111}, we obtain
\begin{equation*}
\begin{split}
&  \sum\limits_{i=1}^n\left\Vert
\Delta\uu_i(\cdot,t)\right\Vert_{r} \leq C\sum\limits_{i=1}^n
\big( \left\Vert\mathbf{f}_i(\cdot,t)
\right\Vert_{r}+\Vert\nabla
\partial_t\mathbf{v}_i(\cdot,t)\Vert_{2}+
 \left\Vert
\Delta\mathbf{v}_i(\cdot,t) \right\Vert_{r} \big),
\end{split}
\end{equation*}
 for $d<r<q\leq 2^*$ and for all $t\in(0,T)$. Taking into account
 \eqref{eq.2.07} and \eqref{eq.4.05}, we arrive at
\begin{equation}\label{eq.4.112}
\begin{split}
& \sum\limits_{i=1}^n \left\Vert\Delta\uu_i(\cdot,t)
\right\Vert_{r} \leq C\sum\limits_{i=1}^n \left( 1+\left\Vert
\Delta\mathbf{v}_i(\cdot,t) \right\Vert_{r}\right).
 \end{split}
\end{equation}
As a consequence,
\begin{equation*}
\begin{split}
& \sum\limits_{i=1}^n \left\Vert\Delta\uu_i(\cdot,t)
\right\Vert_{r}^{r} \leq C\sum\limits_{i=1}^n \big(
1+\left\Vert \Delta\mathbf{v}_i(\cdot,t)
\right\Vert_{r}^{r}\big),\quad d<q<r.
 \end{split}
\end{equation*}
  Next, using \eqref{eq.4.09} and \eqref{eq.1.25},  we can derive
 \begin{equation*}
 \frac{\partial
 \Delta\mathbf{v}_i}{\partial t}
= \sum_{j=1}^{n}a_{ij}\left(-\Delta\mathbf{v}_{j}+\Delta\uu_{j}\right),
\qquad \displaystyle\mathbf{A}=\{a_{ij}\}_{i, j = 1}^{n}:=\mathbf{K}^{-1}\mathbf{M}
\end{equation*}
for all $i\in\{1,\dots,n\}$.
 Using the representation of the solution to the system of ordinary
differential equations, we obtain
\begin{equation*}
\Delta \mathbf{v}_i
=e^{t\mathbf{A}}\Delta \mathbf{v}_i(0)+e^{t\mathbf{A}}\int_{0}^{t}e^{-s\mathbf{A}} {\sum_{j=1}^{n}}a_{ij}\Delta \uu_{j}ds.
\end{equation*}
Applying the Young and H\"older inequalities, together with \eqref{eq.1.24}--\eqref{eq.1.25}, we arrive at
\begin{equation}\label{eq.4.123}
\sum_{i=1}^{n}\left\Vert\Delta \mathbf{v}_i(\cdot,t)\right\Vert_{r}^{r} \leq C
\left(\sum_{i=1}^{n}\left\Vert\Delta \mathbf{v}_{i,0}\right\Vert_{r}^{r} +t^{r-1} \int_{0}^{t}\sum_{i=1}^{n}\left\Vert \Delta \uu_i(\cdot,s)\right\Vert_{r}^{r}\,ds\right).
\end{equation}
Plugging \eqref{eq.4.123} into \eqref{eq.4.112}, we obtain
 \begin{equation}\label{eq.4.114}
 \sum\limits_{i=1}^n \left\Vert\Delta\uu_i(\cdot,t)
\right\Vert_{r}^{r} \leq C(T) \left( 1+
 \int_{0}^{t}\sum\limits_{i=1}^n\left\Vert \Delta \uu_i
(\cdot,s)
\right\Vert_{r}^{r}\,ds\right).
\end{equation}

Application of Gronwall's inequality, \eqref{eq.4.09} and \eqref{eq.4.04} to \eqref{eq.4.114}, leads to the estimate
\begin{equation}\label{eq.4.114+}
 \sup_{t\in[0,T]}\sum\limits_{i=1}^n \left\Vert\Delta\uu_i(\cdot,t)
\right\Vert_{r}^{r} \leq
C(T)\left(1+\sum_{i=1}^{n}\left\Vert\Delta\uu_i(\cdot,0)
\right\Vert_{r}^{r}\right)=K'_{3}.
\end{equation}
Consequently, by \eqref{eq.4.123}, \eqref{eq.4.114+}, and \eqref{eq.4.04}, we have
\begin{equation*}
 \sup_{t\in[0,T]} \sum_{i=1}^{n}\left\Vert\Delta\mathbf{v}_i(\cdot,t)
\right\Vert_{r}^{r} \leq
C(T)\left(1+\sum_{i=1}^{n}\left\Vert\Delta\mathbf{v}_{i,0}
\right\Vert_{r}^{r}\right)\leq C(T)\left(1+K'_{3}\right):=K_3,
\end{equation*}
which gives \eqref{eq.4.03}.
\end{proof}

Now, using Lemmas \ref{lem.4.10} and \ref{lemma:LS}, we prove the following higher regularity properties of the solutions  $(\mathbf{v}_i,\rho_i)$ to the problem
\eqref{eq.1.26}--\eqref{eq.1.30}.
\begin{theorem}
\label{th:A1} Let $\Omega$ be a bounded domain of $\mathbb{R}^d$
with its boundary $\partial\Omega$ supposed to be of class $\C^2$.
Let $\{(\rho_1,\mathbf{v}_1),\dots, (\rho_n,\mathbf{v}_n)\}$ be a weak solution to
the problem \eqref{eq.1.26}--\eqref{eq.1.30} and assume all the conditions of Theorem~\ref{Theorem.2.1} and Lemmas ~\ref{lemma:LS} and \ref{lem.4.10} are satisfied.
 Then for every $i\in\{1,\dots,n\}$ there hold
\begin{equation}\label{eq.4.07}
\begin{split}
& \sup_{t\in\left[0,T\right]}\left(\left\Vert
\Delta\mathbf{v}_i(\cdot,t)\right\Vert_{r}^2 + \left\Vert\nabla
\mathbf{v}_i(\cdot,t)\right\Vert_{0,\alpha}^2\right)+
\int_0^T\left\Vert\nabla {\pi}_i(\cdot,t)\right\Vert_{r}^{2}dt \leq
C_i
\end{split}
\end{equation}
for some positive constant $C_i$ {{and $\alpha$ such that
$0<\alpha\leq1-\frac{d}{r}$}}, and
\begin{equation}\label{eq.4.08}
\sup_{t\in[0,T]}\left(\|\nabla
\rho_i(\cdot,t)\|_{\infty}^2+\|\partial_t\rho_i(\cdot,t)\|_{\infty}^2\right)\leq
C_i
\end{equation}
for another positive constant $C_i$.
\end{theorem}
Recall that $\|\cdot\|_{0,\alpha}$ in \eqref{eq.4.07} denotes the norm in the H\"older function space $\CC^{0,\alpha}(\overline{\Omega})$.

\begin{proof}
Before we get into the details of the proof of
Theorem~\ref{th:A1}, we note that assumption $d<r<2^\ast$ in
\eqref{eq.4.05} above implies the following continuous embedding
\begin{equation*}
\L^{2}\left(0,T;\LL^r(\Omega)\right) \hookrightarrow
\L^{2}\left(0,T;\LL^{2}(\Omega)\right),
  \end{equation*}
and thus assumption \eqref{eq.2.03} can be replaced by
\eqref{eq.4.05} in Theorems~\ref{Theorem.2.1} and~\ref{thm:p}. By
the Sobolev embedding theorem, we can see that assumption
\eqref{eq.2.01} in those results can in turn be replaced by
\eqref{eq.4.04}.  Arguing as in \cite{Antontsev-2021}, combining
\eqref{eq.4.03} with the Sobolev embedding $
\WW^{1,r}_0(\Omega)\hookrightarrow
\mathbf{C}^{0,\alpha}(\overline{\Omega})$, valid for
$0<\alpha\leq1-\frac{d}{r}$,  one has
\begin{equation}\label{eq.4.17}
\left\Vert\nabla
\mathbf{v}_i(\cdot,t)\right\Vert_{0,\alpha}^2\leq
K_3\qquad \forall\ t\in[0,T], \qquad  r>d,
\end{equation}
which, in particular, implies
\begin{equation}\label{eq.4.18}
\int_{0}^{t}\|\nabla \mathbf{v}_i(\cdot,t)\|_{\infty}dt\leq
C(T,K_3),\qquad r>d.
\end{equation}
Thanks to \eqref{eq.4.18}, \eqref{eq.4.17}, and
Theorem~\ref{Theorem.2.1}, by  Lemma~\ref{eq.4.01} with
\eqref{eq.4.02} and \eqref{eq.4.03}, we obtain the following inequalities
\begin{alignat}{2}
& \label{eq.4.19}\|\nabla \rho_i(\cdot,t)\|_{\infty}^2\leq
d\|\nabla \rho_{i,0}\|_{\infty}^2
\exp\left(2T\sqrt{K_3}\right)<\infty, && \\
& \label{eq.4.20} \|\partial_t\rho_i(\cdot,t)\|_{\infty}^2\leq
d\|\nabla \rho_{i,0}\|_{\infty}^2 K_3
\exp\left(2T\sqrt{K_3}\right)<\infty. &&
\end{alignat}
At last, adding up \eqref{eq.4.19} and \eqref{eq.4.20}, and observing
the definition of $K_3$ (see \eqref{eq.4.07}, we prove
\eqref{eq.4.08}, which finishes the proof.
\end{proof}

\section{Uniqueness}\label{Sect:Unique}
Under the regularity properties \eqref{eq.4.07} and \eqref{eq.4.08}, we establish a uniqueness result in the following theorem. For one-constituent fluid, this theorem was proved in
\cite{Antontsev-2021}.
\begin{theorem}\label{thm:uniq}
Let $\{(\rho_1^{(1)},\mathbf{v}_1^{(1)}),\dots,
(\rho_n^{(1)},\mathbf{v}_n^{(1)})\}$ and
$\{(\rho_1^{(2)},\mathbf{v}_1^{(2)}),\dots, (\rho_n^{(2)},\mathbf{v}_n^{(2)})\}$
 be two solutions to the problem
\eqref{eq.1.26}--\eqref{eq.1.30} in the conditions of Theorems
\ref{Theorem.2.1} and \ref{th:A1}. Then
$\mathbf{v}_i^{(1)}=\mathbf{v}_i^{(2)}$ and $\rho_i^{(1)}=\rho_i^{(2)}$
for all $i\in\{1,\dots,n\}$.
 \end{theorem}
\begin{proof}
Let $\{(\rho_1^{(1)},\mathbf{v}_1^{(1)}),\dots,
(\rho_n^{(1)},\mathbf{v}_n^{(1)})\}$ and
$\{(\rho_1^{(2)},\mathbf{v}_1^{(2)}),\dots, (\rho_n^{(2)},\mathbf{v}_n^{(2)})\}$
be two solutions to the problem \eqref{eq.1.26}--\eqref{eq.1.30}
in the conditions of Theorems \ref{Theorem.2.1} and \ref{th:A1}.
We set $\mathbf{v}_i=\mathbf{v}_i^{(1)}-\mathbf{v}_i^{(2)}$ and
$\rho_i=\rho_i^{(1)}-\rho_i^{(2)}$ for all $i\in\{1,\dots,n\}$.

By subtracting the equations \eqref{eq.1.26}, \eqref{eq.1.27}, and
\eqref{eq.1.28} for $(\mathbf{v}_i^{(1)},\pi_i^{(1)},\rho_i^{(1)})$ to
the
 equations \eqref{eq.1.26}, \eqref{eq.1.27}, and \eqref{eq.1.28}
 for $(\mathbf{v}_i^{(2)}, \pi_i^{(2)}, \rho_i^{(2)})$, we obtain
\begin{align}
&\mathrm{div}\mathbf{v}_i=0,\nonumber
\\
\label{eq.6.02}%
\begin{split}
&\partial_t\left(\rho_i^{(1)}\mathbf{v}_i\right) +
\mathbf{div}\left(\rho_i^{(1)}\mathbf{v}_i^{(1)}\otimes\mathbf{v}_i\right)
=\rho_i\mathbf{f}_i-\nabla\pi_i-\rho_i\partial_t\mathbf{v}^{(2)}_i
+\mathrm{div}\left(\rho_i\mathbf{v}^{(1)}_i\right)\mathbf{v}^{(2)}_i
\\
& +\mathrm{div}\left(\rho^{(2)}_i\mathbf{v}_i\right)\mathbf{v}^{(2)}_i
-\mathbf{div}\left(\rho_i^{(1)}\mathbf{v}_i\otimes\mathbf{v}^{(2)}_i\right)
-\mathbf{div}\left(\rho_i\mathbf{v}^{(2)}_i\otimes\mathbf{v}^{(2)}_i\right)
\\
& +\sum\limits_{j=1}^n\Big(\mu_{ij}\Delta\mathbf{v}_j
+\kappa_{ij}\partial_t\Delta\mathbf{v}_j
+\gamma_{ij}\left(\mathbf{v}_j-\mathbf{v}_i\right)\Big),
\end{split}
\\
& \label{eq.6.03}%
\partial_t\rho_i+\left(\mathbf{v}^{(1)}_i\cdot\nabla\right)\rho_i+\left(\mathbf{v}_i \cdot\nabla\right)\rho^{(2)}_i=0,
\end{align}
where $i=1,\ldots,n.$

Multiplying \eqref{eq.6.02} by $2\mathbf{v}_i$, integrating the resulting
equation over $\Omega$, summing up by $i$ from $1$ to $n$, and using
the formula
\begin{equation*}
2\partial_t\left(\rho^{(1)}_i\mathbf{v}_i\right)\mathbf{v}_i=\partial_t\rho^{(1)}_i
|\mathbf{v}_i|^2+
\partial_t\left(\rho^{(1)}_i|\mathbf{v}_i|^2\right),
\end{equation*}
we arrive at
\begin{equation}\label{eq.6.04}%
\begin{split}
&
\frac{d}{dt}\left(\sum\limits_{i=1}^n\int_{\Omega}\rho^{(1)}_i|\mathbf{v}_i|^2
\,d\mathbf{x} +
\sum\limits_{i,j=1}^n\kappa_{ij}\int_{\Omega}\nabla\mathbf{v}_j:\nabla\mathbf{v}_i
\,d\mathbf{x} \right)
\\&
+2\sum\limits_{i,j=1}^n\mu_{ij}\int_{\Omega}\nabla\mathbf{v}_j:\nabla\mathbf{v}_i\,d\mathbf{x}
+\sum\limits_{i=1}^n\int_{\Omega}\partial_t\rho^{(1)}_i|\mathbf{v}_i|^2
\,d\mathbf{x}
\\&
+2\sum\limits_{i=1}^n\int_{\Omega}\mathbf{div}\left(\rho^{(1)}_i\mathbf{v}_i^{(1)}\otimes\mathbf{v}_i\right)\cdot\mathbf{v}_i\,d\mathbf{x}
=2\sum\limits_{i=1}^{n}\int_\Omega
\rho_i\mathbf{f}_i\cdot\mathbf{v}_i\,d\mathbf{x}
\\&
+2\sum\limits_{i=1}^{n}\int_\Omega
\mathrm{div}\left(\rho_i\mathbf{v}^{(1)}_i\right)\mathbf{v}^{(2)}_i\cdot\mathbf{v}_i\,d\mathbf{x}
\\&
+2\sum\limits_{i=1}^{n}\int_\Omega
\mathrm{div}\left(\rho^{(2)}_i\mathbf{v}_i\right)\mathbf{v}^{(2)}_i
\cdot\mathbf{v}_i\,d\mathbf{x}
\\
&-2\sum\limits_{i=1}^n\int_{\Omega}\mathbf{div}\left(\rho^{(1)}_i\mathbf{v}_i\otimes\mathbf{v}_i^{(2)}\right)\cdot\mathbf{v}_i\,d\mathbf{x}
\\
& -2\sum\limits_{i=1}^n\int_{\Omega}\mathbf{div}\left(\rho_i
\mathbf{v}_i^{(2)}\otimes\mathbf{v}_i^{(2)}\right)\cdot\mathbf{v}_i\,d\mathbf{x}
\\
& -\sum\limits_{i=1}^{n}\int_\Omega
\rho_i\partial_t\mathbf{v}_i^{(2)}\cdot\mathbf{v}_i\,d\mathbf{x}-\sum\limits_{i,j=1}^{n}\gamma_{ij}\int_{\Omega}\left|\mathbf{v}_i-\mathbf{v}_j\right|^2\,d\mathbf{x}.
\end{split}
\end{equation}
Arguing as we did for \eqref{eq.A}--\eqref{eq.AA}, one has
\begin{equation*}
\sum\limits_{i=1}^n\int_{\Omega}\partial_t\rho^{(1)}_i|\mathbf{v}_i|^2
\,d\mathbf{x}+2\sum\limits_{i=1}^n\int_{\Omega}\mathbf{div}\left(\rho^{(1)}_i\mathbf{v}_i^{(1)}\otimes\mathbf{v}_i\right)\cdot\mathbf{v}_i\,d\mathbf{x}=0.
\end{equation*}
Moreover, using
\begin{align*}
\begin{split}
&\mathbf{div}\left(\rho^{(1)}_i\mathbf{v}_i\otimes\mathbf{v}_i^{(2)}\right)\cdot\mathbf{v}_i=
\mathrm{div}\left(\rho^{(1)}_i\mathbf{v}_i\right)\mathbf{v}_i^{(2)}\cdot\mathbf{v}_i
+\rho^{(1)}_i(\mathbf{v}_i\cdot\nabla)\mathbf{v}_i^{(2)}\cdot\mathbf{v}_i,
\end{split}
\\
\begin{split}
&\mathbf{div}\left(\rho_i\mathbf{v}_i^{(2)}\otimes\mathbf{v}_i^{(2)}\right)\cdot\mathbf{v}_i=
\mathrm{div}\left(\rho_i\mathbf{v}_i^{(2)}\right)\mathbf{v}_i^{(2)}\cdot\mathbf{v}_i
+\rho_i(\mathbf{v}_i^{(2)}\cdot\nabla)\mathbf{v}_i^{(2)}\cdot\mathbf{v}_i,
\end{split}
\end{align*}
we get
\begin{equation*}
\begin{split}
&\int_\Omega
\mathrm{div}\left(\rho_i\mathbf{v}^{(1)}_i\right)\mathbf{v}^{(2)}_i\cdot\mathbf{v}_i\,d\mathbf{x}
+\int_\Omega \mathrm{div}\left(\rho^{(2)}_i\mathbf{v}_i\right)\mathbf{v}^{(2)}_i
\cdot\mathbf{v}_i\,d\mathbf{x}
\\
&-\int_{\Omega}\mathbf{div}
(\rho^{(1)}_i\mathbf{v}_i\otimes\mathbf{v}_i^{(2)})\cdot\mathbf{v}_i\,d\mathbf{x}
-\int_{\Omega}\mathbf{div}\left(\rho_i\mathbf{v}_i^{(2)}
\otimes\mathbf{v}_i^{(2)}\right)\cdot\mathbf{v}_i\,d\mathbf{x}=
\\
& \int_{\Omega}\mathbf{div}\left(\rho_i\mathbf{v}_i^{(1)}-\rho^{(1)}_i\mathbf{v}_i+\rho^{(2)}_i\mathbf{v}_i-\rho_i\mathbf{v}^{(2)}_i\right)\mathbf{v}_i^{(2)}
\cdot\mathbf{v}_i\,d\mathbf{x}
\\
&-\int_\Omega
\rho^{(1)}_i(\mathbf{v}_i\cdot\nabla)\mathbf{v}_i^{(2)}\cdot\mathbf{v}_i\,d\mathbf{x}-\int_\Omega
\rho_i\left(\mathbf{v}_i^{(2)}\cdot\nabla\right)\mathbf{v}_i^{(2)}\cdot\mathbf{v}_i\,d\mathbf{x}=
\\
&-\int_\Omega
\rho^{(1)}_i(\mathbf{v}_i\cdot\nabla)\mathbf{v}_i^{(2)}\cdot\mathbf{v}_i\,d\mathbf{x}-\int_\Omega
\rho_i\left(\mathbf{v}_i^{(2)}\cdot\nabla\right)\mathbf{v}_i^{(2)}\cdot\mathbf{v}_i\,d\mathbf{x}.
\end{split}
\end{equation*}
Plugging into \eqref{eq.1.28},
we obtain
\begin{equation*}%
\begin{split}
&
\frac{d}{dt}\left(\sum\limits_{i=1}^n\int_{\Omega}\rho^{(1)}_i|\mathbf{v}_i|^2
\,d\mathbf{x} +
\sum\limits_{i,j=1}^n\kappa_{ij}\int_{\Omega}\nabla\mathbf{v}_j:\nabla\mathbf{v}_i\,d\mathbf{x} \right)
\\&
+2\sum\limits_{i,j=1}^n\mu_{ij}\int_{\Omega}\nabla\mathbf{v}_j:\nabla\mathbf{v}_i
\,d\mathbf{x}
=2\sum\limits_{i=1}^{n}\int_\Omega\rho_i\mathbf{f}_i
\cdot\mathbf{v}_i\,d\mathbf{x}
\\&
-\sum\limits_{i=1}^{n}\int_\Omega\rho_i\partial_t\mathbf{v}_i^{(2)}\cdot\mathbf{v}_i
\,d\mathbf{x}
-2\sum\limits_{i=1}^{n}\int_\Omega\rho_i\left(\mathbf{v}_i^{(2)}
\cdot\nabla\right)\mathbf{v}_i^{(2)}\cdot\mathbf{v}_i\,d\mathbf{x}
\\&
-2\sum\limits_{i=1}^{n}\int_\Omega
\rho^{(1)}_i(\mathbf{v}_i\cdot\nabla)\mathbf{v}_i^{(2)}\cdot\mathbf{v}_i\,d\mathbf{x} -
\sum\limits_{i,j=1}^{n}\gamma_{ij}\int_{\Omega}\left|\mathbf{v}_i-\mathbf{v}_j\right|^2\,d\mathbf{x}.
\end{split}
\end{equation*}
Then, multiplying \eqref{eq.6.03} by  $2\rho_i$, integrating the
resulting equation over $\Omega$  and using
\[
2\int_{\Omega}\rho_i(\mathbf{v}_i^{(1)}\cdot\nabla)\rho_i\,d\mathbf{x}
=\int_{\Omega}\mathbf{v}_i^{(1)}\cdot\nabla
\rho_i^2\,d\mathbf{x}=-
\int_{\Omega}\rho_i^2\mathrm{div}\mathbf{v}_i^{(1)}\,d\mathbf{x}=0
\]
we obtain
\begin{equation}
\label{eq.6.05}%
\begin{split}
& \frac{d}{dt}\int_{\Omega}|\rho_i|^2 \,d\mathbf{x}
=-2\int_{\Omega}\left(\left(\mathbf{v}^{(1)}_i \cdot\nabla\right)
\rho_i+(\mathbf{v}_i \cdot\nabla)\rho^{(2)}_i\right)\rho_i\,d\mathbf{x} =-2\int_{\Omega}\rho_i \mathbf{v}_i\cdot\nabla\rho^{(2)}_i \,d\mathbf{x}.
\end{split}
\end{equation}
Summing \eqref{eq.6.05} by $i$ from $1$ to $n$ and add with
($\ref{eq.6.04}$). Then integrating the result by $s$ from $0$ to
$t$, we obtain
\begin{equation*}
\begin{split}
& \sum\limits_{i=1}^n\left(\|\sqrt{\rho^{(1)}_i}\mathbf{v}_i
\|^2_{2}+\|\rho_i\|^2_{2}\right) +
\sum\limits_{i,j=1}^n\kappa_{ij}\int_{\Omega}\nabla\mathbf{v}_j:\nabla\mathbf{v}_i
\,d\mathbf{x}
+2\int_0^t\sum\limits_{i,j=1}^n\mu_{ij}\int_{\Omega}\nabla\mathbf{v}_j:\nabla\mathbf{v}_i
\,d\mathbf{x} ds
\\&
=2\int_0^t\sum\limits_{i=1}^{n}\int_\Omega
\rho_i\mathbf{f}_i\cdot\mathbf{v}_i\,d\mathbf{x}
ds-\int_0^t\sum\limits_{i=1}^{n}\int_\Omega
\rho_i\partial_t\mathbf{v}_i^{(2)}\cdot\mathbf{v}_i\,d\mathbf{x} ds
\\&
-2\int_0^t\sum\limits_{i=1}^{n}\int_\Omega
\rho_i\left(\mathbf{v}_i^{(2)}\cdot\nabla\right)\mathbf{v}_i^{(2)}\cdot\mathbf{v}_i\,d\mathbf{x} ds
-2\int_0^t\sum\limits_{i=1}^{n}\int_\Omega
\rho^{(1)}_i(\mathbf{v}_i\cdot\nabla)\mathbf{v}_i^{(2)}\cdot\mathbf{v}_i\,d\mathbf{x} ds
\\&
-2\int_0^t\sum\limits_{i,j=1}^{n}\int_{\Omega}\rho_i \mathbf{v}_i
\cdot\nabla\rho^{(2)}_i \,d\mathbf{x} ds
-\int_0^t\sum\limits_{i,j=1}^{n}\gamma_{ij}\int_{\Omega}\left|\mathbf{v}_i-\mathbf{v}_j\right|^2\,
d\mathbf{x} ds,
\end{split}
\end{equation*}
for all $t\in[0,T]$.
Similarly as we did for \eqref{eq.2.32}--\eqref{eq.2.37}, we get
\begin{equation}
\label{eq.6.06}%
\begin{split}
& y(t)+
2\mu^-\int_0^t\sum\limits_{i=1}^n\|\nabla\mathbf{v}_i(\cdot,s)\|^2_{2}\,ds
\leq 2\int_0^t\sum\limits_{i=1}^{n}\int_\Omega
\rho_i\mathbf{f}_i\cdot\mathbf{v}_i\,d\mathbf{x} ds
\\&
-\int_0^t\sum\limits_{i=1}^{n}\int_\Omega
\rho_i\partial_t\mathbf{v}_i^{(2)}\cdot\mathbf{v}_i\,d\mathbf{x} ds
\\&
-2\int_0^t\sum\limits_{i=1}^{n}\int_\Omega
\rho_i\left(\mathbf{v}_i^{(2)}\cdot\nabla\right)\mathbf{v}_i^{(2)}\cdot\mathbf{v}_i\,d\mathbf{x}
ds
\\&
-2\int_0^t\sum\limits_{i=1}^{n}\int_\Omega
\rho^{(1)}_i(\mathbf{v}_i\cdot\nabla)\mathbf{v}_i^{(2)}\cdot\mathbf{v}_i\,d\mathbf{x}
ds
\\&
-2\int_0^t\sum\limits_{i,j=1}^{n}\int_{\Omega}\rho_i\mathbf{v}_i
\cdot\nabla\rho^{(2)}_i \,d\mathbf{x} ds:= \sum_{k=1}^{5}I_k(t),
\end{split}
\end{equation}
where
\[
y(t)=\sum\limits_{i=1}^n\big(\rho^-\|\mathbf{v}_i(\cdot,t)\|^2_{2}+\|\rho_i(\cdot,t)\|^2_{2}\big)
+ \kappa^-\sum\limits_{i=1}^n\|\nabla\mathbf{v}_i(\cdot,t)\|^2_{2}.
\]

Now, we estimate each term on the right-hand side of
\eqref{eq.6.06}.
Using the H\"{o}lder, Cauchy, and Sobolev
inequalities, together with the estimates \eqref{eq.2.06} and
\eqref{eq.2.07}, we obtain the following estimates for $I_k$,
$k\in\{1\dots,5\}$, for all $t\in[0,T]$:
\begin{align}
\label{eq.6.07}
\begin{split}
|I_1(t)|&\leq2\sum\limits_{i=1}^{n}\int_0^t\int_\Omega |
\rho_i\mathbf{f}_i\cdot\mathbf{v}_i|\,d\mathbf{x} ds
\leq2\int_0^t\sum\limits_{i=1}^{n}\|\rho_i(\cdot,s)\|_{2}\|\mathbf{f}_i(\cdot,s)\|_{r}\|\mathbf{v}_i(\cdot,s)\|_{q}
\,ds
\\&
\leq2\int_0^t\sum\limits_{i=1}^{n}\|\rho_i(\cdot,s)\|_{2}\|\mathbf{f}_i(\cdot,s)\|_{r}C(\Omega)\|\nabla\mathbf{v}_i(\cdot,s)\|_{2}\,ds
\\&
\leq
C(\Omega,\kappa^-)\int_0^t\|\mathbf{f}_i(\cdot,s)\|_{r}\sum\limits_{i=1}^{n}\Big(\|\rho_i(\cdot,s)\|^2_{2}+\kappa^-\|\nabla\mathbf{v}_i(\cdot,s)\|^2_{2}
\Big)\,ds
\\
&\leq \int_0^t C_1(s) y(s)ds, \quad
\frac{1}{2}+\frac{1}{r}+\frac{1}{q}=1,\quad 1\leq
q\leq\frac{2d}{d-2},
\end{split}
\end{align}
where $C_1(t)=C(\Omega, \kappa^-)\|\mathbf{f}_i(\cdot,t)\|_{r}$
and $C_1\in \L^1([0,T])$.
Analogously,
\begin{align}
\label{eq.6.08}%
\begin{split}
&|I_2(t)|\leq\int_0^t \sum\limits_{i=1}^{n}\int_\Omega
|\rho_i(\mathbf{x},s)|\,
|\partial_t\mathbf{v}_i^{(2)}(\mathbf{x},s)|\,
| \mathbf{v}_i(\mathbf{x},s)|\,d\mathbf{x} ds
\\&
\leq
\int_0^t\sum\limits_{i=1}^{n}\|\rho_i(\cdot,s)\|_{2}
\|\partial_t\mathbf{v}_i^{(2)}(\cdot,s)\|_{4}\|\mathbf{v}_i(\cdot,s)\|_{4}\,ds
\\&
\leq C_2(\Omega, \kappa^-)
\int_0^t\sum\limits_{i=1}^{n}\|\partial_t\nabla\mathbf{v}_i^{(2)}(\cdot,s)\|_{2}\Big(\|\rho_i(\cdot,s)\|^2_{2}+
\kappa^-\|\nabla \mathbf{v}_i(\cdot,s)\|^2_{2}\Big)\,ds
\\&
\leq C_2 \int_0^t y(s)ds, \quad {d\leq4},
\end{split}
\end{align}
\begin{align}
\label{eq.6.09}%
\begin{split}
|I_3(t)|&\leq 2\int_0^t\sum\limits_{i=1}^{n}\int_\Omega |\rho_i(\mathbf{x},s)|\,
|\mathbf{v}_i^{(2)}(\mathbf{x},s)|\, |\nabla\mathbf{v}_i^{(2)}(\mathbf{x},s)|\, |\mathbf{v}_i(\mathbf{x},s)|\,d\mathbf{x} ds
\\&
\leq 2\int_0^t\sum\limits_{i=1}^{n}
\|\mathbf{v}^{(2)}_i(\cdot,s)\|_{4}\|\nabla
\mathbf{v}^{(2)}_i(\cdot,s)\|_{\infty}
 \|\rho_i(\cdot,s)\|_{2}\|\mathbf{v}_i(\cdot,s)\|_{4}\,ds
\\
&\leq C(\Omega)\int_0^t\sum\limits_{i=1}^{n}
\|\nabla\mathbf{v}^{(2)}_i(\cdot,s)\|_{2}\|\nabla
\mathbf{v}^{(2)}_i(\cdot,s)\|_{\infty}\|\rho_i(\cdot,s)\|_{2}\|\nabla\mathbf{v}_i(\cdot,s)\|_{2}
\,ds
\\&
\leq  C_3(\kappa^-, \Omega, K_3)\int_0^t \sum\limits_{i=1}^{n}
\Big(\|\rho_i(\cdot,s)\|_{2}^2+\kappa^-\|\nabla
\mathbf{v}_i(\cdot,s)\|^2_{2}\Big)\,ds \leq C_3\int_0^ty(s)\,ds,
\end{split}
\end{align}
\begin{align}
\begin{split}
|I_4(t)|&\leq 2\int_0^t\sum\limits_{i=1}^{n}\int_\Omega
|\rho^{(1)}_i(\mathbf{v}_i\cdot\nabla)\mathbf{v}_i^{(2)}\cdot\mathbf{v}_i|\,d\mathbf{x} ds
\\&
\leq 2\rho^+\int_0^t\sum\limits_{i=1}^{n} \|\nabla
\mathbf{v}^{(2)}_i(\cdot,s)\|_{2} \|\mathbf{v}_i(\cdot,s)\|^2_{4}\,ds
\\&
\leq C(\rho^+,\Omega,\kappa^-)\int_0^t\sum\limits_{i=1}^{n}
\kappa^-\|\nabla\mathbf{v}_i(\cdot,s)\|^2_{2}\,ds\leq C_4\int_0^ty(s)\,ds,
\end{split}
\label{eq.6.10}
\end{align}
\begin{align}
\label{eq.6.11}%
\begin{split}
|I_5(t)|\leq & 2\int_0^t\sum\limits_{i,j=1}^{n}\int_{\Omega}|\rho_i
\mathbf{v}_i \cdot\nabla\rho^{(2)}_i| \,d\mathbf{x} ds
\\&
\leq\int_0^t \sum\limits_{i=1}^{n} \|\mathbf{v}_i(\cdot,s)\|_{2}\|\nabla
\rho^{(2)}_i(\cdot,s)\|_{\infty}
 \|\rho_i(\cdot,s)\|_{2}\,ds
\\&
\leq  C_5(\rho^-, K_3)\int_0^t \sum\limits_{i=1}^{n}
\Big(\rho^-\|\mathbf{v}_i(\cdot,s)\|_{2}^2+\|\rho_i(\cdot,s)\|_{2}^2\Big)\,ds
\\&
\leq C_5 \int_0^ty(s)\,ds.
\end{split}
\end{align}

Note that in the estimates of $I_3$ and $I_5$ obtained in
\eqref{eq.6.09} and \eqref{eq.6.11}, we have used  the estimates  \eqref{eq.4.18} and \eqref{eq.4.19}, respectively.
Plugging \eqref{eq.6.07}--\eqref{eq.6.11} into \eqref{eq.6.06}, we
arrive at the integral inequality
  \begin{equation}\label{eq.6.12}%
 y(t)\leq \int_0^t a(s)y(s)\,ds,
\end{equation}
where $a(t)=C_1(t)+C_2+C_3+C_4+C_5$. It follows from \eqref{eq.6.12}
and  {Gr\"{o}nwall}'s inequality  that
\begin{equation*}
y(t)=\sum\limits_{i=1}^n\Big(\rho^-\|\mathbf{v}_i(\cdot,t)\|^2_{2}+\|\rho_i(\cdot,t)\|^2_{2}\Big)
+ \kappa^-\sum\limits_{i=1}^n\|\nabla\mathbf{v}_i(\cdot,t)\|^2_{2}=0,\quad
\forall\ t\in[0,T],
\end{equation*}
which yields $\mathbf{v}_i^{(1)}=\mathbf{v}_i^{(2)}$ and $\rho_i^{(1)}=\rho_i^{(2)}$ for all $i\in\{1,\dots,n\}$.
\end{proof}

\section*{Acknowledgments}
S.N.~Antontsev, I.V.~Kuznetsov and D.A.~Prokudin were supported by the Ministry of Science and Higher
Education of the Russian Federation under project no. FWGG-2021-0010, Russian Federation.
I.V.~Kuznetsov and D.A.~Prokudin were supported by the Ministry of Science and Higher
Education of the Russian Federation under project no. FZMW-2024-0003, Russian Federation.
H.B.~de~Oliveira and Kh.~Khompysh were supported by the Grant no. AP19676624 from the Ministry of Science and Education of the Republic of Kazakhstan.
H.B.~de~Oliveira was also partially supported by CIDMA under the Portuguese Foundation for Science and Technology
Multi-Annual Financing Program for R\&D Units (FCT, \url{https://ror.org/00snfqn58}).
\bibliographystyle{AMS}
\bibliography{ADKKP-11-06-25}

\begin{bibdiv}
\begin{biblist}

\bib{AOKh:2019:b}{article}{
      author={Antontsev, S.~N.},
      author={de~Oliveira, H.~B.},
      author={Khompysh, Kh.},
       title={Existence and large time behavior for generalized
  {K}elvin--{V}oigt equations governing nonhomogeneous and incompressible
  fluids},
        date={2019},
     journal={J. Phys. Conf. Ser.},
      volume={1268},
       pages={012008},
         url={https://doi.org/10.1088/1742-6596/1268/1/012008},
}

\bib{AOKh:2019:c}{article}{
      author={Antontsev, S.~N.},
      author={de~Oliveira, H.~B.},
      author={Khompysh, Kh.},
       title={Generalized {K}elvin--{V}oigt equations for nonhomogeneous and
  incompressible fluids},
        date={2019},
     journal={Commun. Math. Sci.},
      volume={17},
      number={7},
       pages={1915\ndash 1948},
         url={https://doi.org/10.4310/CMS.2019.v17.n7.a3},
}

\bib{AOKh:2020}{article}{
      author={Antontsev, S.~N.},
      author={de~Oliveira, H.~B.},
      author={Khompysh, Kh.},
       title={{K}elvin--{V}oigt equations with anisotropic diffusion,
  relaxation, and damping: blow-up and large time behavior},
        date={2020},
     journal={Asymptot. Anal.},
      volume={121},
      number={2},
       pages={125\ndash 157},
         url={https://doi.org/10.3233/ASY-201638},
}

\bib{Antontsev-2021}{article}{
      author={Antontsev, S.~N.},
      author={de~Oliveira, H.~B.},
      author={Khompysh, Kh.},
       title={The classical {K}elvin--{V}oigt problem for incompressible fluids
  with unknown non-constant density: existence, uniqueness and regularity},
        date={2021},
     journal={Nonlinearity},
      volume={34},
       pages={3083\ndash 3111},
         url={https://doi.org/10.1088/1361-6544/abe159},
}

\bib{AKPS}{article}{
      author={Antontsev, S.~N.},
      author={Kuznetsov, I.~V.},
      author={Prokudin, D.~A.},
      author={Sazhenkov, S.~A.},
       title={The impulsive {K}elvin--{V}oigt equations for two-constituent
  mixtures},
        date={2025},
     journal={Sib Electr Math Rep},
      volume={22},
      number={1},
       pages={563\ndash 586},
         url={https://doi.org/10.33048/semi.2025.22.038},
}

\bib{AC:1976a}{article}{
      author={Atkin, R.~J.},
      author={Craine, R.~E.},
       title={Continuum theories of mixtures: Applications},
        date={1976},
     journal={IMA J. Appl. Math.},
      volume={17},
      number={2},
       pages={153\ndash 207},
  url={https://academic.oup.com/imamat/article-abstract/17/2/153/707579},
}

\bib{AC:1976b}{article}{
      author={Atkin, R.~J.},
      author={Craine, R.~E.},
       title={Continuum theories of mixtures: Basic theory and historical
  development},
        date={1976},
     journal={Q. J. Mech. Appl. Math.},
      volume={29},
      number={2},
       pages={209\ndash 244},
  url={https://academic.oup.com/qjmam/article-abstract/29/2/209/1849954},
}

\bib{Baris:2013}{article}{
      author={Baris, S.},
      author={Demir, M.~S.},
       title={Hydromagnetic flows of a mixture of two {N}ewtonian fluids
  between two parallel plates},
        date={2013},
        ISSN={1992-1950},
     journal={Int. J. Phys. Sci.},
      volume={8},
      number={9},
       pages={343\ndash 355},
}

\bib{Barnes:2000}{book}{
      author={Barnes, H.~A.},
       title={A handbook of elementary rheology},
   publisher={University of Wales, Institute of Non-Newtonian Fluid Mechanics},
     address={Aberystwyth},
        date={2000},
        ISBN={0-9538032-0-1},
         url={https://archive.org/details/handbookelementa00barn},
}

\bib{Beevers:1982}{article}{
      author={Beevers, C.~E.},
      author={Craine, R.~E.},
       title={On the determination of response functions for a binary mixture
  of incompressible {N}ewtonian fluids},
        date={1982},
     journal={Int. J. Eng. Sci.},
      volume={20},
      number={6},
       pages={737\ndash 745},
  url={https://www.sciencedirect.com/science/article/pii/0020722582900830},
}

\bib{BAH:1987a}{book}{
      author={Bird, R.~B.},
      author={Armstrong, R.~C.},
      author={Hassager, O.},
       title={Dynamics of polymeric liquids, volume 1: Fluid mechanics},
     edition={2},
   publisher={Wiley-Interscience},
     address={New York},
        date={1987},
        ISBN={978-0-471-83128-9},
}

\bib{Bogovskii:1980}{article}{
      author={Bogovski{\u\i}, M.~E.},
       title={Solutions of some problems of vector analysis, associated with
  the operators div and grad},
        date={1980},
     journal={Trudy Sem. S. L. Soboleva},
      number={1},
       pages={5\ndash 40},
        note={In Russian},
}

\bib{CM:2016}{article}{
      author={Cianchi, A.},
      author={Maz'ya, V.},
       title={Sobolev inequalities in arbitrary domains},
        date={2016},
     journal={Adv. Math.},
      volume={293},
       pages={644\ndash 696},
  url={https://www.sciencedirect.com/science/article/pii/S0001870815301572},
}

\bib{CL:1955}{book}{
      author={Coddington, E.~A.},
      author={Levinson, N.},
       title={Theory of ordinary differential equations},
   publisher={McGraw-Hill Book Company},
     address={New York},
        date={1955},
        ISBN={978-0-07-011542-2},
}

\bib{CF:1988}{book}{
      author={Constantin, P.},
      author={Foias, C.},
       title={{N}avier-{S}tokes equations},
      series={Chicago Lectures in Mathematics},
   publisher={University of Chicago Press},
     address={Chicago},
        date={1988},
        ISBN={978-0-226-11549-8},
  url={https://press.uchicago.edu/ucp/books/book/chicago/N/bo5973146.html},
}

\bib{CSS:2002}{article}{
      author={Cotter, C.~S.},
      author={Smolarkiewicz, P.~K.},
      author={Szczyrba, I.~N.},
       title={A viscoelastic fluid model for brain injuries},
        date={2002},
     journal={Int. J. Numer. Methods Fluids},
      volume={40},
      number={1--2},
       pages={303\ndash 311},
         url={https://onlinelibrary.wiley.com/doi/10.1002/fld.287},
}

\bib{deOliveira}{article}{
      author={de~Oliveira, H.~B.},
      author={Khompysh, Kh.},
      author={Shakir, A.},
       title={Strong solutions for the {N}avier--{S}tokes--{V}oigt equations
  with non-negative density},
        date={2025},
     journal={J. Math. Phys.},
      volume={66},
       pages={041506},
         url={https://doi.org/10.1063/5.0217358},
}

\bib{Frehse_SJMA_2005}{article}{
      author={Frehse, J.},
      author={Goj, S.},
      author={M\'alek, J.},
       title={On a {S}tokes-like system for mixtures of fluids},
        date={2005},
     journal={SIAM J. Math. Anal.},
      volume={36},
      number={4},
       pages={1259\ndash 1281},
         url={https://epubs.siam.org/doi/10.1137/S0036141003429151},
}

\bib{Frehse_AM_2005}{article}{
      author={Frehse, J.},
      author={Goj, S.},
      author={M\'alek, J.},
       title={A uniqueness result for a model for mixtures in the absence of
  external forces and interaction momentum},
        date={2005},
     journal={Appl. Math.},
      volume={50},
      number={6},
       pages={527\ndash 541},
         url={https://link.springer.com/article/10.1007/s10492-005-0035-x},
}

\bib{Frehse_AM_2008}{article}{
      author={Frehse, J.},
      author={Weigant, W.},
       title={On quasi-stationary models of mixtures of compressible fluids},
        date={2008},
     journal={Appl. Math.},
      volume={53},
      number={4},
       pages={319\ndash 345},
         url={https://link.springer.com/article/10.1007/s10492-008-0029-6},
}

\bib{Galdi:2011}{book}{
      author={Galdi, G.~P.},
       title={An introduction to the mathematical theory of the
  {N}avier-{S}tokes equations: Steady-state problems},
     edition={2},
      series={Springer Monographs in Mathematics},
   publisher={Springer},
     address={New York},
        date={2011},
        ISBN={978-0-387-09619-3},
         url={https://link.springer.com/book/10.1007/978-0-387-09620-9},
}

\bib{Temam:2023}{article}{
      author={Giorgini, A.},
      author={Ndongmo~Ngana, A.},
      author={Tachim~Medjo, T.},
      author={Temam, R.},
       title={Existence and regularity of strong solutions to a nonhomogeneous
  {K}elvin--{V}oigt--{C}ahn--{H}illiard system},
        date={2023-11},
        ISSN={0022-0396},
     journal={Journal of Differential Equations},
      volume={372},
       pages={612\ndash 656},
         url={http://dx.doi.org/10.1016/j.jde.2023.07.024},
}

\bib{HMPST:2017}{article}{
      author={Hron, J.},
      author={Milo\v{s}, V.},
      author={Pru\v{s}a, V.},
      author={Sou\v{c}ek, O.},
      author={Tuma, K.},
       title={On thermodynamics of incompressible viscoelastic rate type fluids
  with temperature dependent material coefficients},
        date={2017},
     journal={Int. J. Non-Linear Mech.},
      volume={95},
       pages={193\ndash 208},
  url={https://www.sciencedirect.com/science/article/pii/S0020746217302329},
}

\bib{Kirwan.Fluids:2020}{article}{
      author={Kirwan, A.~D.},
      author={Massoudi, M.},
       title={The heat flux vector(s) in a two component fluid mixture},
        date={2020},
     journal={Fluids},
      volume={5},
      number={2},
       pages={77},
         url={https://www.mdpi.com/2311-5521/5/2/77},
}

\bib{LS:1978}{article}{
      author={Ladyzhenskaya, O.~A.},
      author={Solonnikov, V.~A.},
       title={Unique solvability of an initial- and boundary-value problem for
  viscous incompressible nonhomogeneous fluids},
        date={1978},
     journal={J. Soviet Math.},
      volume={9},
       pages={697\ndash 749},
         url={https://link.springer.com/article/10.1007/BF01085325},
}

\bib{Mamontov:2014}{article}{
      author={Mamontov, A.~E.},
      author={Prokudin, D.~A.},
       title={Solubility of a stationary boundary-value problem for the
  equations of motion of a one-temperature mixture of viscous compressible
  heat-conducting fluids},
        date={2014},
     journal={Izv. Math.},
      volume={78},
      number={3},
       pages={554\ndash 579},
         url={https://www.mathnet.ru/eng/im8507},
}

\bib{Mamontov:2021}{article}{
      author={Mamontov, A.~E.},
      author={Prokudin, D.~A.},
       title={Solubility of unsteady equations of the three-dimensional motion
  of two-constituent viscous compressible heat-conducting fluids},
        date={2021},
     journal={Izv. Math.},
      volume={85},
      number={4},
       pages={755\ndash 812},
         url={https://www.mathnet.ru/eng/im8507},
}

\bib{Massoudi:2008}{article}{
      author={Massoudi, M.},
       title={A note on the meaning of mixture viscosity using the classical
  continuum theories of mixtures},
        date={2008},
     journal={Int. J. Eng. Sci.},
      volume={46},
      number={7},
       pages={677\ndash 689},
  url={https://www.sciencedirect.com/science/article/pii/S002072250800013X},
}

\bib{Mazya:2011}{book}{
      author={Maz'ya, V.},
       title={Sobolev spaces with applications to elliptic partial differential
  equations},
      series={Grundlehren der mathematischen Wissenschaften},
   publisher={Springer},
     address={Berlin, Heidelberg},
        date={2011},
      volume={342},
        ISBN={978-3-642-15563-5},
         url={https://link.springer.com/book/10.1007/978-3-642-15564-2},
}

\bib{Mohan:2020}{article}{
      author={Mohan, M.~T.},
       title={On the three dimensional {K}elvin--{V}oigt fluids: global
  solvability, exponential stability and exact controllability of {G}alerkin
  approximations},
        date={2020},
     journal={Evol. Equ. Control Theory},
      volume={9},
      number={2},
       pages={301\ndash 339},
         url={https://www.aimsciences.org/article/doi/10.3934/eect.2020007},
}

\bib{Oldroyd:1950}{article}{
      author={Oldroyd, J.~G.},
       title={On the formulation of rheological equations of state},
        date={1950},
     journal={Proc. R. Soc. Lond. A},
      volume={200},
      number={1063},
       pages={523\ndash 541},
         url={https://royalsocietypublishing.org/doi/10.1098/rspa.1950.0035},
}

\bib{Oskolkov:1985}{article}{
      author={Oskolkov, A.~P.},
       title={Theory of nonstationary flows of {K}elvin--{V}oigt fluids},
        date={1985},
     journal={J. Soviet Math.},
      volume={28},
       pages={751\ndash 758},
         url={https://doi.org/10.1007/BF02112340},
}

\bib{Oskolkov:1989}{article}{
      author={Oskolkov, A.~P.},
       title={Initial-boundary value problems for equations of motion of
  {K}elvin--{V}oigt fluids and oldroyd fluids},
        date={1989},
     journal={Proc. Steklov Inst. Math.},
      volume={179},
       pages={137\ndash 182},
         url={https://doi.org/10.1070/SM1989v064n01ABEH003325},
}

\bib{Pav:1971}{article}{
      author={Pavlovsky, V.~A.},
       title={On the theoretical description of weak water solutions of
  polymers},
        date={1971},
     journal={Dokl. Akad. Nauk SSSR},
      volume={200},
      number={4},
       pages={809\ndash 812},
}

\bib{Pileckas:1980}{article}{
      author={Pileckas, K.~I.},
       title={Spaces of solenoidal vectors},
        date={1984},
     journal={Proc. Steklov Inst. Math.},
      volume={159},
       pages={141\ndash 154},
}

\bib{Prokudin:2021}{article}{
      author={Prokudin, D.~A.},
       title={Existence of weak solutions to the problem on three-dimensional
  steady heat-conductive motions of compressible viscous multicomponent
  mixtures},
        date={2021},
     journal={Sib. Math. J.},
      volume={62},
      number={5},
       pages={895\ndash 907},
         url={https://doi.org/10.1134/S0037446621050111},
}

\bib{Raja-Tao:1995}{book}{
      author={Rajagopal, K.~R.},
      author={Tao, L.},
       title={Mechanics of mixtures},
      series={Advances in Mathematics for Applied Sciences},
   publisher={World Scientific Publishing Co., Inc.},
     address={River Edge, NJ},
        date={1995},
      volume={35},
         url={https://doi.org/10.1142/2197},
}

\bib{SW:1977}{article}{
      author={Sampaio, R.},
      author={Williams, W.~O.},
       title={On the viscosities of liquid mixtures},
        date={1977},
     journal={Z. Angew. Math. Phys.},
      volume={28},
      number={4},
       pages={607\ndash 614},
         url={https://doi.org/10.1007/BF01601339},
}

\bib{Simon:1990}{article}{
      author={Simon, J.},
       title={Nonhomogeneous viscous incompressible fluids: existence of
  velocity, density, and pressure},
        date={1990},
     journal={SIAM J. Math. Anal.},
      volume={21},
      number={5},
       pages={1093\ndash 1117},
         url={https://doi.org/10.1137/0521061},
}

\bib{Temam:1984}{book}{
      author={Temam, R.},
       title={{N}avier-{S}tokes equations: Theory and numerical analysis},
      series={Studies in Mathematics and its Applications},
   publisher={North-Holland Publishing Co.},
     address={Amsterdam},
        date={1984},
      volume={2},
}

\bib{EZS:2024}{article}{
      author={ten Eikelder, M. F.~P.},
      author={van~der Zee, K.~G.},
      author={Schillinger, D.},
       title={Thermodynamically consistent diffuse-interface mixture models of
  incompressible multicomponent fluids},
        date={2024},
     journal={J. Fluid Mech.},
      volume={990},
       pages={A8},
         url={https://doi.org/10.1017/jfm.2024.8},
}

\bib{Truesdell:1984}{book}{
      author={Truesdell, C.},
       title={Rational thermodynamics},
     edition={2},
   publisher={Springer},
        date={1984},
}

\bib{Turbin:2023}{article}{
      author={Turbin, M.},
      author={Ustiuzhaninova, A.},
       title={Existence of weak solution to initial-boundary value problem for
  finite order {K}elvin--{V}oigt fluid motion model},
        date={2023},
     journal={Bol. Soc. Mat. Mex.},
      volume={29},
       pages={54},
         url={https://doi.org/10.1007/s40590-023-00454-7},
}

\bib{Zvyagin-2018-1}{article}{
      author={Zvyagin, A.~V.},
       title={Weak solvability of {K}elvin--{V}oigt model of
  thermoviscoelasticity},
        date={2018},
     journal={Russ. Math.},
      volume={62},
      number={3},
       pages={79\ndash 83},
         url={https://doi.org/10.3103/S1066369X18030099},
}

\bib{Zvyagin:2024}{article}{
      author={Zvyagin, A.~V.},
       title={On the existence of weak solutions of the {K}elvin--{V}oigt
  model},
        date={2024},
     journal={Math. Notes},
      volume={116},
      number={1},
       pages={130\ndash 135},
         url={https://doi.org/10.1134/S0001434624010154},
}

\bib{Zvyagin2010}{article}{
      author={Zvyagin, V.~G.},
      author={Turbin, M.~V.},
       title={The study of initial-boundary value problems for mathematical
  models of the motion of {K}elvin--{V}oigt fluids},
        date={2010},
     journal={J. Math. Sci.},
      volume={168},
       pages={157\ndash 308},
         url={https://doi.org/10.1007/s10958-010-9981-2},
}

\bib{Zvyagin:2023}{article}{
      author={Zvyagin, V.~G.},
      author={Turbin, M.~V.},
       title={An existence theorem for weak solutions of the initial-boundary
  value problem for the nonhomogeneous incompressible {K}elvin--{V}oigt model
  in which the initial value of density is not bounded from below},
        date={2023},
     journal={Math. Notes},
      volume={114},
      number={1},
       pages={630\ndash 634},
         url={https://doi.org/10.1134/S0001434623030157},
}

\end{biblist}
\end{bibdiv}

\end{document}